\documentclass[12pt,reqno]{amsart}

\usepackage{graphicx}
\usepackage{amssymb}
\usepackage{amsmath}
\usepackage{amscd}
\usepackage[all]{xy}
\usepackage[margin=2.9cm]{geometry}
\usepackage[svgnames]{xcolor}

\author{Sergiy Maksymenko}
\title[Fundamental groups of orbits]
{Structure of the fundamental groups of orbits of smooth functions on surfaces}
\address{Topology dept. \\ Institute of Mathematics of NAS of Ukraine \\ Te\-re\-shchen\-kivska st. 3, Kyiv, 01601 Ukraine}
\email{maks@imath.kiev.ua}
\urladdr{http://www.imath.kiev.ua/~maks}


\newtheorem{theorem}[subsection]{Theorem}
\newtheorem{lemma}[subsection]{Lemma}

\newtheorem{proposition}[subsection]{Proposition}

\newtheorem{remark}[subsection]{Remark}
\newtheorem*{remark*}{Remark}

\newtheorem{definition}[subsection]{Definition}

\newenvironment{axiom}[1]
{
\par\smallskip\noindent
{\bf Axiom #1.}\begin{it}
}
{
\end{it}\smallskip
 }

\makeatletter
\@addtoreset{equation}{section}
\@addtoreset{figure}{section}
\@addtoreset{table}{section}
\makeatother

\makeatletter
\newcommand\testshape{family=\f@family; series=\f@series; shape=\f@shape.}
\def\myemphInternal#1{\if n\f@shape%
\begingroup\itshape #1\endgroup\/%
\else\begingroup\bfseries #1\endgroup%
\fi}
\def\myemph{\futurelet\testchar\MaybeOptArgmyemph}
\def\MaybeOptArgmyemph{\ifx[\testchar \let\next\OptArgmyemph
                 \else \let\next\NoOptArgmyemph \fi \next}
\def\OptArgmyemph[#1]#2{\index{#1}\myemphInternal{#2}}
\def\NoOptArgmyemph#1{\myemphInternal{#1}}
\makeatother

\newcommand\RRR{{\mathbb R}}

\newcommand\ZZZ{{\mathbb Z}}

\newcommand\FFF{{\mathcal F}}

\newcommand\FF{{\mathcal F}}

\newcommand\eps{\varepsilon}

\newcommand\id{\mathrm{id}}          
\newcommand\Int{\mathrm{Int}}        
\newcommand\wrm[1]{\mathop{\wr}\limits_{\ZZZ_{#1}}}

\newcommand\AxCrPt{{\text{\rm(L)}}}
\newcommand\AxBd{{\text{\rm(B)}}}

\newcommand\Bman{B}
\newcommand\Mman{M}
\newcommand\Nman{N}
\newcommand\Pman{P}

\newcommand\Uman{U}
\newcommand\Vman{V}
\newcommand\Wman{W}
\newcommand\Xman{X}
\newcommand\Yman{Y}
\newcommand\Zman{Z}

\newcommand\Circle{S^1}            

\newcommand\Disk{D^2}              
\newcommand\MobiusBand{Mo}         

\newcommand\Sphere{S^2}            
\newcommand\Torus{T^2}             

\newcommand\Orb{\mathcal{O}}        
\newcommand\Stab{\mathcal{S}}       
\newcommand\Diff{\mathcal{D}}       
\newcommand\Aut{\mathrm{Aut}}       
\newcommand\Maps[2]{#2^{#1}}        

\newcommand\DiffId{\Diff_{\id}}     
\newcommand\StabId{\Stab_{\id}}     

\newcommand\Ci[2]{\mathcal{C}^{\infty}(#1,#2)}               
\newcommand\Cid[2]{\mathcal{C}_{\partial}^{\infty}(#1,#2)}   

\newcommand\func{f}

\newcommand\dif{h}
\newcommand\hdif{\widehat{\dif}}
\newcommand\gdif{g}
\newcommand\hgdif{\widehat{\gdif}}

\newcommand\DiffM{\Diff(\Mman)}

\newcommand\DiffIdM{\DiffId(\Mman)}

\newcommand\DiffMX{\Diff(\Mman, \Xman)}
\newcommand\DiffIdMX{\DiffId(\Mman, \Xman)}

\newcommand\Morse{\mathrm{Morse}}

\newcommand\aFlow{\mathbf{F}}
\newcommand\classP{\mathcal{P}}
\newcommand\classR{\mathcal{R}}

\newcommand\aGrp{A}
\newcommand\bGrp{B}

\newcommand\Stabilizer[1]{\Stab(#1)}             
\newcommand\StabilizerId[1]{\StabId(#1)}         
\newcommand\StabilizerIsotId[1]{\Stab'(#1)}      
\newcommand\StabilizerInv[1]{\Stab_{\mathrm{inv}}(#1)}  

\newcommand\Orbit[1]{\Orb(#1)}                   
\newcommand\OrbitPathComp[2]{\Orb_{#2}(#1)}      

\newcommand\SingularSet[1]{\Sigma_{#1}}             

\newcommand\KronrodReebGraph[1]{\Gamma(#1)}                   
\newcommand\AutKRGraphStab[1]{\mathbf{G}}             

\newcommand\fStab{\Stabilizer{\func}}             
\newcommand\fStabId{\StabilizerId{\func}}         
\newcommand\fStabIsotId{\StabilizerIsotId{\func}}  

\newcommand\fStabX{\Stabilizer{\func,\Xman}}             
\newcommand\fStabIdX{\StabilizerId{\func, \Xman}}         
\newcommand\fStabIsotIdX{\StabilizerIsotId{\func, \Xman}}  

\newcommand\fOrb{\Orbit{\func}}                     
\newcommand\fOrbComp{\OrbitPathComp{\func}{\func}}  

\newcommand\fOrbX{\Orbit{\func, \Xman}}                     
\newcommand\fOrbCompX{\OrbitPathComp{\func, \Xman}{\func}}  

\newcommand\fSing{\SingularSet{\func}}                

\newcommand\fKRGraph{\KronrodReebGraph{\func}}            
\newcommand\fKRAut{\AutKRGraphStab{\func}(\func)}         
\newcommand\homStabToAutG{\lambda}     

\newcommand\hMman{\widehat{\Mman}}  

\newcommand\crLev{K}
\newcommand\crNbh{\Nman}

\newcommand\iSurf{\Xman}
\newcommand\nSurf{\Yman}

\newcommand\hiSurf{S}
\newcommand\hnSurf{T}

\newcommand\usim{\underset{\crLev}{\sim}}
\newcommand\dsim{\underset{\partial\crNbh}{\sim}}

\newcommand\bComp{B}
\newcommand\hbComp{\widehat{\bComp}}
\newcommand\hcrNbh{\widehat{\crNbh}}
\newcommand\bCyl{C}
\newcommand\dCyl{Q}
\newcommand\hdCyl{\widetilde{\dCyl}}

\newcommand\uTriv{\mathcal{T}(\func,\crLev)}
\newcommand\dTriv{\mathcal{T}(\func,\partial\crNbh)}

\newcommand\isoStabfB{\psi}
\newcommand{\qind}{q}

\begin{document}

\begin{abstract}
Let $M$ be a smooth compact connected surface, $P$ be either the real line $\mathbb{R}$ or the circle $S^1$ and $f:M\to P$ be a Morse map.
Denote by $\mathcal{S}(f)$ and $\mathcal{O}(f)$ the corresponding stabilizer and orbit of $f$ with respect to the right action of the group $\mathcal{D}(M)$ of diffeomorphisms of $M$.
In a series of papers the author described homotopy types of $\mathcal{S}(f)$ and computed higher homotopy groups of $\mathcal{O}(f)$.
The present paper describes the structure of the remained fundamental group $\pi_1 \mathcal{O}(f)$ for the case when $M$ is orientable and differs from $2$-sphere and $2$-torus.

The result holds as well for a larger class of smooth maps $f:M\to P$ having the following property: the germ of $f$ at each of its critical points is smoothly equivalent to a homogeneous polynomial $\mathbb{R}^2\to\mathbb{R}$ without multiple factors.
\end{abstract}

\maketitle

\section{Introduction}
Let $\Mman$ be a smooth compact connected surface and $\Pman$ be either the real line $\RRR$ or the circle $\Circle$.
For each closed subset $\Xman\subset\Mman$ let $\DiffMX$ be the group of $C^{\infty}$-diffeomorphisms fixed on $\Xman$ and 
\begin{align*}
\fStabX &= \{\dif\in\DiffMX \mid \func\circ \dif=\func\}, 
& 
\fOrbX &=\{\func\circ\dif\mid \dif\in\DiffMX\}
\end{align*}
be respectively the \textit{stabilizer} and the \textit{orbit} of $\func\in C^{\infty}(\Mman,\Pman)$ under the standard right action of $\DiffMX$ on $C^{\infty}(\Mman,\Pman)$.

We will endow $\DiffMX$ and $\Ci{\Mman}{\Pman}$ with $C^{\infty}$ Whitney topologies.
These topologies induce certain topologies on the spaces $\fStabX$ and $\fOrbX$.
Denote by $\DiffIdMX$ and $\fStabIdX$ the identity path components of $\DiffMX$ and $\fStabX$ respectively and by $\fOrbComp$ the path component of $\fOrb$ containing $\func$.
If $\Xman=\varnothing$, we will omit $\Xman$ from notation, e.g.\! we write $\DiffM$ instead of $\Diff(\Mman,\varnothing)$, and so on.

In~\cite{Maksymenko:AGAG:2006, Maksymenko:MFAT:2009, Maksymenko:MFAT:2010, Maksymenko:ProcIM:ENG:2010, Maksymenko:UMZ:ENG:2012} for a large class of smooth maps $\func:\Mman\to\Pman$ and certain ``$\func$-adopted submanifolds'' $\Xman\subset\Mman$ the author described the homotopy types of $\fStabX$, computed the higher homotopy groups of $\fOrbX$, and obtained certain information about $\pi_1\fOrbX$, see Theorem~\ref{th:right_action_props} below.

The main result of this paper, Theorem~\ref{th:structure_of_pi1_fOrb}, gives a complete description of the structure of $\pi_1\fOrbX$ for the case when $\Mman$ is orientable and differs from $2$-sphere and $2$-torus.
It expresses $\pi_1\fOrbX$ in terms of special wreath product with $\ZZZ$ over some finite cyclic groups.

\medskip

\subsection{Preliminaries}
Let $\Cid{\Mman}{\Pman}$ be the subset of $\Ci{\Mman}{\Pman}$ consisting of maps $\func$ satisfying the following axiom:
\begin{axiom}{\AxBd}
The map $\func:\Mman\to\Pman$ takes constant values on connected components of $\partial\Mman$, and the set $\fSing$ of critical points of $\func$ 
is contained in the interior $\Int{\Mman}$.
\end{axiom}

Denote by $\Morse(\Mman,\Pman) \subset \Cid{\Mman}{\Pman}$ the subset consisting of \myemph{Morse} maps, i.e.\! maps having only non-degenerate critical points.
It is well known that $\Morse(\Mman,\Pman)$ is open and everywhere dense in $\Cid{\Mman}{\Pman}$.

Let also $\FF(\Mman,\Pman)$ be the subset of $\Cid{\Mman}{\Pman}$ consisting of maps $\func$ satisfying the following additional axiom:
\begin{axiom}{\AxCrPt}
For each critical point $z$ of $\func$ the germ of $\func$ at $z$ is smoothly equivalent to some homogeneous polynomial $\func_z:\RRR^2\to\RRR$ without multiple factors.
\end{axiom}
By Morse lemma a non-degenerate critical point of a map $\func:\Mman\to\Pman$ is smoothly equivalent to a homogeneous polynomial $\pm x^2\pm y^2$ which, evidently, has no multiple factors, and so satisfies $\AxCrPt$.
This implies that
\[
 \Morse(\Mman,\Pman) \ \subset \ \FF(\Mman,\Pman). 
\]

Notice that every critical point satisfying Axiom~\AxCrPt\ is isolated.
Moreover such a point $z$ is non-degenerate if and only if the corresponding homogeneous polynomial $\func_z$ has degree $\geq 3$, see e.g.~\cite[\S7]{Maksymenko:MFAT:2009}.

\begin{definition}\label{defn:f_adopted_subsurface}{\rm\cite{Maksymenko:UMZ:ENG:2012}}.
Let $\func\in\FFF(\Mman,\Pman)$, $\Xman \subset \Mman$ be a compact submanifold, not necessarily connected, and whose connected components may have distinct dimensions. 
Let also $\Xman^{i}$, $i=0,1,2$, be the union of connected components of $\Xman$ of dimension $i$.
Then $\Xman$ will be called \myemph{$\func$-adopted} if the following conditions hold true:
\begin{enumerate}
\item[(a)]
$\Xman\cap\fSing \subset \Xman^{0}\cup \Int\Xman^2$;
\item[(b)]
$\func$ takes constant value at each connected component of $\Xman^{1} \cup \partial\Xman^{2}$.
\end{enumerate}
\end{definition}

The following lemma gives examples of $\func$-adopted submanifolds.
We left it to the reader.
\begin{lemma}
Let $\Xman, \Yman\subset\Mman$ be two submanifolds.

{\rm 1)} If $\Xman$ is $\func$-adopted, then so is every connected component of $\Xman$.

{\rm 2)} If $\Xman$ and $\Yman$ are $\func$-adopted and disjoint, then $\Xman\cup\Yman$ is $\func$-adopted as well.

{\rm 3)} Suppose every connected component of $\Xman$ has dimension $2$ю
Then $\Xman$ is $\func$-adopted if and only if the restriction $\func|_{\Xman}$ satisfies axioms \AxBd\ and \AxCrPt.

{\rm 4)} Let $a,b\in\Pman$ be two distinct regular values of $\func\in\FFF(\Mman,\Pman)$, and $[a,b] \subset \Pman$ the closed segment between them.
Then $\Xman = \func^{-1}[a,b]$ and any family of connected components of $\Xman$ is $\func$-adopted.
\qed
\end{lemma}

Let $\func\in\FFF(\Mman,\Pman)$ and $\Xman \subset \Mman$ be an $\func$-adopted submanifold.
Denote 
\begin{align}\label{equ:StabIsotId}
\fStabIsotId &= \fStab \cap \DiffIdM,
&
\fStabIsotIdX &= \fStab \cap \DiffIdMX.
\end{align}
In a sequel all the homotopy groups of $\Orbit{\func,\Xman}$ will have $\func$ as a base point, and so the notation $\pi_k\fOrbX$ will always mean $\pi_k(\fOrbX, \func)$.
Notice that the latter group is also isomorphic with $\pi_k(\OrbitPathComp{\func,\Xman}{\func}, \func)$.
The following theorem summarizes the results concerning the homotopy types of $\fStabIdX$ and $\fOrbCompX$.

\begin{theorem}\label{th:right_action_props}{\rm \cite{Maksymenko:AGAG:2006, Maksymenko:MFAT:2010, Maksymenko:ProcIM:ENG:2010, Maksymenko:UMZ:ENG:2012}.}
Let $\func\in\FFF(\Mman,\Pman)$ and $\Xman \subset \Mman$ be an $\func$-adopted submanifold.
Then the following statements hold true.

{\rm 1)}
$\fOrbCompX = \OrbitPathComp{\func, \Xman\cup\partial\Mman}{\func}$.

{\rm 2)}
The map $p:\DiffMX \longrightarrow \fOrbX$ defined by $p(\dif) = \func \circ \dif$ is a Serre fibration.

{\rm 3)}
Suppose that either $\func$ has at least one critical point being {\bfseries not a non-degenerate local extreme} or $\Mman$ is non-orientable.
Then $\fStabId$ is contractible, $\pi_n\fOrb = \pi_n\Mman$ for $n\geq3$, $\pi_2\fOrbComp=0$, and we also have the following short exact sequence:
\begin{equation}\label{equ:pi1Of_exact_sequence}
 1 \longrightarrow \pi_1\DiffIdM \xrightarrow{~~~p~~~} \pi_1\fOrb \xrightarrow{~~~\partial~~~} \pi_0\fStabIsotId \longrightarrow 1.
\end{equation}
If $\Mman$ is orientable and distinct from $\Sphere$ and $\Torus$ then $\pi_1\fOrb$ is solvable.

{\rm 4)}
Suppose that the Euler characteristic $\chi(\Mman)$ is less than the number of points in $\Xman$.
This holds for instance when either $\dim\Xman>0$ or $\chi(\Mman)<0$.
Then both $\DiffIdMX$ and $\fStabIdX$ are contractible, $\pi_i\fOrbX=0$ for $i\geq2$, and we have an isomorphism:
\[
 \pi_1\Orbit{\func, \Xman} \ \cong \ \pi_0\StabilizerIsotId{\func,\Xman}.
\]
Moreover, there exist finitely many mutually disjoint $\func$-adopted subsurfaces $\Bman_1,\ldots,\Bman_n \subset \Mman\setminus(\Xman^{1} \cup \Xman^{2})$ each diffeomorphic either to a $2$-disk $\Disk$, or a cylinder $S^1\times I$, or a M\"obius band $\MobiusBand$, and such that if we denote
\begin{align*}
\func_i &:= \func|_{\Bman_i}: \Bman_i \to \Pman,  &
\hat{\Bman_i} &:= (\Bman_i\cap\Xman^0) \cup \partial\Bman_i
\end{align*}
for $i=1,\ldots,n$, then the following isomorphisms hold:
\[
\pi_1\fOrbCompX \ \cong \ \pi_0\fStabIsotIdX \ \cong \
\mathop{\times}_{i=1}^{n} \pi_0\StabilizerIsotId{\func_i, \,\hat{\Bman_i}\,} \ \cong \
\mathop{\times}_{i=1}^{n} \pi_0\OrbitPathComp{\func_i, \,\hat{\Bman_i}\,}{\func_i}.
\]

{\rm 5)}
Let $\Uman$ be any open neighbourhood of $\Xman^1\cup \Xman^2$.
Then there exists an $\func$-adopted submanifold $\Nman\subset\Mman$ such that every connected component of $\Nman$ has dimension $2$, 
\begin{align*}
&\Xman^0 \cap\Nman = \varnothing,&
&\Xman^1\cup \Xman^2 \subset \Int{\Nman} \subset\Nman \subset \Uman
\end{align*}
and the inclusion $\StabilizerIsotId{\func,\Xman^0 \cup \Nman} \subset\StabilizerIsotId{\func, \Xman}$ is a homotopy equivalence.
\end{theorem}
\begin{proof}
Statements 1) and 5) are proved in~\cite[Corollaries~2.1 \& 6.2]{Maksymenko:UMZ:ENG:2012} respectively.

Statement 2) is a general result initially established in the paper by F.~Sergeraert~\cite{Sergeraert:ASENS:1972} for smooth functions of \textit{finite codimension} on arbitrary closed manifolds.
In particular, all singularities satisfying Axiom~\AxCrPt\ have finite codimensions, \cite[Lemma~12]{Maksymenko:ProcIM:ENG:2010}.
This covers the case $\Xman=\varnothing$.
The proof for $\Xman=\fSing$ was given in~\cite[\S11]{Maksymenko:AGAG:2006}, and for arbitrary $\func$-adopted submanifold $\Xman$ in~\cite[Theorem~5.1]{Maksymenko:UMZ:ENG:2012}.

Statement 3) is proved in~\cite[Theorems~1.3, 1.5]{Maksymenko:AGAG:2006} for Morse maps, and extended to the class $\FFF(\Mman,\Pman)$ in~\cite{Maksymenko:ProcIM:ENG:2010}.
Solvability result is obtained in~\cite{Maksymenko:KRGraphs:2013}.

Statement 4) was initially established in~\cite[Theorem~1.7]{Maksymenko:MFAT:2010} for $\Xman=\varnothing$, and extended to the general case in~\cite[Theorem~2.4]{Maksymenko:UMZ:ENG:2012}.
\end{proof}

\subsection{Wreath products $A\wrm{m}\ZZZ$ and $A\wr\ZZZ_m$}
Let $A$ be any group and $\ZZZ_m$, $m\geq1$, be a finite cyclic group of order $m$.
Denote by $\Maps{\ZZZ_m}{A}$ the set of \textit{all maps} $\ZZZ_m\to A$ (being not necessarily homomorphisms).
Then $\ZZZ_m$ naturally acts on $\Maps{\ZZZ_m}{A}$ from the right and therefore one can define the corresponding semidirect product $\Maps{\ZZZ_m}{A} \rtimes \ZZZ_m$ which is denoted by \[A \wr \ZZZ_m\] and called \emph{wreath product} of $A$ and $\ZZZ_m$.

\medskip

More generally, notice that the group $\ZZZ$ acts on $\ZZZ_{m}$ by the rule
\[ z * k = z+k \mod m\]
for $z\in \ZZZ$ and $k\in\ZZZ_{m}$.
This action induces a right action of $\ZZZ$ on $\Maps{\ZZZ_m}{A}$ and therefore one can define the corresponding semidirect product $\Maps{\ZZZ_m}{A} \rtimes \ZZZ$ which will be denoted by
\[
\aGrp\wrm{m}\ZZZ
\]
and called the \myemph{wreath product} of $A$ and $\ZZZ$ over $\ZZZ_{m}$. 

Thus $\aGrp\wrm{m}\ZZZ$ is the set $\Maps{\ZZZ_m}{A} \times \ZZZ$ with the multiplication defined as follows.
Let $(\alpha,a), (\beta,b)\in \Maps{\ZZZ_m}{A} \times \ZZZ$.
Define $\gamma: \ZZZ_m \to A$ by the formula $\gamma(i) = \alpha(i+b)\cdot\beta(i)$, $i\in\ZZZ_m$, where $\cdot$ is the multiplication in $A$.
Then, by definition, the product $(\alpha,a)$ and $(\beta,b)$ in $\aGrp\wrm{m}\ZZZ$ is
\[
(\alpha,a) (\beta,b) := (\gamma, a+b).
\]
If $\epsilon: \ZZZ_m \to A$ is the constant map into the unit of $A$, then $(\epsilon, 0)$ is the unit of $\aGrp\wrm{m}\ZZZ$.

Evidently, for $m=1$ the group $\aGrp\wrm{m}\ZZZ$ is isomorphic with the direct product $A \times \ZZZ$.
Also if $A=\{1\}$, then $\aGrp\wrm{m}\ZZZ \cong \ZZZ$ for all $m\geq1$.
Finally notice that there is a natural \textit{epimorphism}
\[
q: A\wrm{m}\ZZZ \longrightarrow A\wr\ZZZ_m,
\qquad
q(\alpha, n) = (\alpha, n\ \mathrm{mod} \ m).
\]

\subsection{Action of $\fStab$ on the Kronrod-Reeb graph}
For $\func\in\FFF(\Mman,\Pman)$ denote by $\fKRGraph$ the {\em Kronrod-Reeb graph} of $\func$, i.e. the factor-space of $\Mman$ obtained by shrinking every connected component of every level-set $\func^{-1}(c)$ of $\func$ to a point.
This graph is very useful for understanding the structure of $\func$, see e.g.~\cite{Kulinich:MFAT:1998, BolsinovFomenko:1997, Kudryavtseva:MatSb:1999, Sharko:UMZ:2003}.

Notice that there is a natural action of $\fStab$ on $\fKRGraph$ defined as follows.
Let $\dif\in\fStab$, so $\func\circ\dif=\func$.
Then $\dif(\func^{-1}(c))=\func^{-1}(c)$ for all $c\in\Pman$.
In particular, $\dif$ interchanges connected components of $\func^{-1}(c)$ being points of $\fKRGraph$, and therefore it yields a certain homeomorphism $\homStabToAutG(\dif)$ of $\fKRGraph$, such that the correspondence $\dif\mapsto\homStabToAutG(\dif)$ is a homomorphism $\homStabToAutG:\fStab\to\Aut(\fKRGraph)$ into the group of all automorphisms of $\fKRGraph$.
Let
\[
\fKRAut :=  \homStabToAutG(\fStabIsotId)
\]
be the group of automorphisms of $\fKRGraph$ induced by isotopic to the identity diffeomorphisms of $\Mman$ preserving $\func$.

\begin{definition}{\rm\cite{Maksymenko:KRGraphs:2013}.}
Let $\classR$ be the minimal class of all finite groups satisfying the following conditions:
\begin{enumerate}
\item
the unit group $\{1\}$ belongs to $\classR$;
\item
if $\aGrp,\bGrp\in\classR$ then $\aGrp\times \bGrp\in\classR$;
\item 
if $\aGrp\in\classR$ and $m\geq1$, then $\aGrp\wr\ZZZ_m\in\classR$.
\end{enumerate}
\end{definition}

\begin{theorem}\label{th:characterization_G}{\rm\cite{Maksymenko:KRGraphs:2013}}.
Let $\Mman$ be a compact orientable surface distinct from $\Sphere$ and $\Torus$.
Then the class $\classR$ coincides with each of the following classes of groups:
\begin{align*}
&\{\, \fKRAut \,\mid\, \func\in\Morse(\Mman,\Pman) \,\}, & 
&\{\, \fKRAut \,\mid\, \func\in\FFF(\Mman,\Pman) \,\}.
\end{align*}
\end{theorem}
In other words, for each $\func \in \FFF(\Mman,\Pman)$ the group $\fKRAut$ can be obtained from the unit group $\{1\}$ by finitely many operations of direct products and wreath products from the top with certain finite cyclic groups.
Conversely, for any group $G\in\classR$ one can find $\func\in\FFF(\Mman,\Pman)$, which can be assumed even Morse, such that $G \cong \fKRAut$.

Since $\fKRAut$ is a finite group, $\homStabToAutG$ reduces an \textit{epimorphism} $\homStabToAutG:\pi_0\fStabIsotId\to\fKRAut$.
Also notice that we have a \textit{surjective} boundary homomorphism $\partial_1: \pi_1\fOrbComp \longrightarrow \pi_0\fStabIsotId$, see Eq.~\eqref{equ:pi1Of_exact_sequence}.
Therefore 
\[
\fKRAut = \homStabToAutG\circ\partial_1(\pi_1\fOrb)
\]
is a factor group of $\pi_1\fOrb$.
Thus Theorem~\ref{th:characterization_G} says that the factor group $\fKRAut$ of $\pi_1\fOrb$ can be described in terms of wreath products $A\wr\ZZZ_m$ being factor groups of $A\wrm{m}\ZZZ$.

\medskip
 
Our main result, Theorem~\ref{th:structure_of_pi1_fOrb} below, shows that $\pi_1\fOrb$ itself can be described in terms of wreath products $A\wrm{m}\ZZZ$.

\begin{definition}
Let $\classP$ be the minimal class of groups satisfying the following conditions:
\begin{enumerate}
\item
the unit group $\{1\}$ belongs to $\classP$;
\item
if $\aGrp,\bGrp\in\classP$, then $\aGrp\times \bGrp\in\classP$;
\item 
if $\aGrp\in\classP$ and $m\geq1$, then $\aGrp\wrm{m}\ZZZ \in\classP$.
\end{enumerate}
\end{definition}

\begin{theorem}\label{th:structure_of_pi1_fOrb}
Let $\Mman$ be a connected compact orientable surface distinct from $2$-sphere and $2$-torus.
Then the class $\classP$ coincides with each of the following two classes of fundamental groups:
\begin{align*}
&\{ \pi_1\fOrb \mid \func\in\Morse(\Mman,\Pman) \}, & 
&\{ \pi_1\fOrb \mid \func\in\FFF(\Mman,\Pman) \}.
\end{align*}
\end{theorem}

It means that for each $\func \in \FFF(\Mman,\Pman)$ the group $\pi_1\fOrb$ can be obtained from $\{1\}$ by finitely many operations of direct products and wreath products from the top with $\ZZZ$ over certain finite cyclic groups.
Conversely, for any group $G\in\classP$ one can find $\func\in\FFF(\Mman,\Pman)$, which can be assumed even Morse, such that $G \cong \pi_1\fOrb$.

The proof will be given in~\S\ref{sect:proof:th:structure_of_pi1_fOrb}.

\begin{remark}\rm
If $\func$ is generic, that is every critical level set of $\func$ contains exactly one critical point, then by~\cite{Maksymenko:AGAG:2006}, $\fKRAut=\{1\}$ and $\pi_1\fOrb \cong \ZZZ^k$ for some $k\geq0$.
In particular, $\pi_1\fOrb\in\classP$.
\end{remark}

\begin{remark}\rm
It is proved in~\cite{Maksymenko:KRGraphs:2013} that each group $A\in\classR$ is solvable.
By similar arguments the same statement can be established for the class $\classP$.
Therefore Theorem~\ref{th:structure_of_pi1_fOrb} gives another proof of solvability result in 3) of Theorem~\ref{th:right_action_props}.
We leave the details for the reader.
\end{remark}

\begin{remark}\rm
Theorem~\ref{th:structure_of_pi1_fOrb} holds for certain classes of smooth functions on $\Torus$, see~\cite{MaksymenkoFeshchenko:UMZ:ENG:2014}. 
\end{remark}

\section{Critical level sets}\label{sect:critical_level_sets}
Let $\func\in\FFF(\Mman,\Pman)$, $c\in\Pman$ and $\crLev$ be a connected component of the level set $\func^{-1}(c)$.
We will call $\crLev$ \emph{critical} if it contains critical points of $\func$, and \emph{regular} otherwise.

Assume that $\crLev$ is critical.
Then due to Axiom~\AxCrPt\ $\crLev$ has a structure of a $1$-dimensional CW-complex whose $0$-cells are critical points of $\func$ belonging to $\crLev$.

Choose $\eps>0$ and let $\crNbh$ be a connected component of $\func^{-1}[c-\eps,c+\eps]$ containing $\crLev$.
We will call $\crNbh$ an \emph{$\func$-regular neighbourhood} of $\crLev$ if the following two conditions hold, see Figure~\ref{fig:crlev_nbh}:
\begin{align*}
&\crNbh \cap\partial\Mman = \varnothing, & 
&\crNbh \cap\fSing = \crLev \cap\fSing.
\end{align*}

\begin{figure}[ht]
\centerline{\includegraphics[width=7cm]{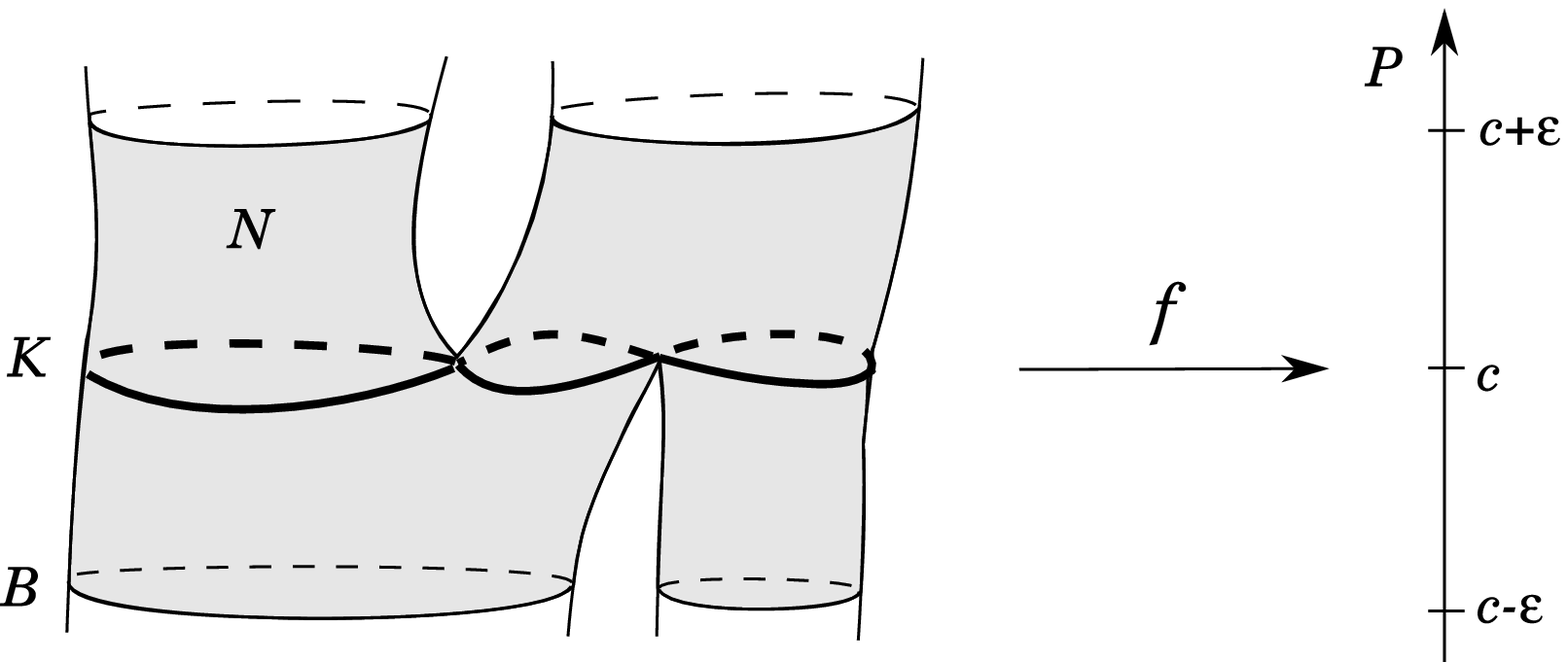}}
\caption{Critical component of level-set of $\func$}
\protect\label{fig:crlev_nbh}
\end{figure}

Let also
\[
\StabilizerInv{\func,\crLev} = \{ \dif\in\Stabilizer{\func} \mid \dif(\crLev) = \crLev\}
\]
be the subgroup of $\Stabilizer{\func}$ consisting of diffeomorphism leaving $\crLev$ invariant.
We will now define two equivalence relations $\usim$ and $\dsim$ on $\StabilizerInv{\func,\crLev}$.

Given a homeomorphism $\dif:\Mman\to\Mman$ and a {\em connected orientable submanifold} $\Xman\subset \Mman$ we will say that $\Xman$ is \textit{positively invariant} for $\dif$ whenever the following conditions hold true:
\begin{itemize}
\item $\dif(\Xman) = \Xman$, and
\item if $\Xman$ is not a point, then the restriction map $\dif|_{\Xman}:\Xman \to \Xman$ is a preserving orientation homeomorphism.
\end{itemize}

Let $\dif\in\StabilizerInv{\func,\crLev}$.
Since $\dif$ preserves the set of critical points of $\func$, it follows that the restriction $\dif|_{\crLev}: \crLev \to \crLev$ is a cellular homeomorphism of $\crLev$.
We will say that $\dif$ is \emph{$\crLev$-trivial}, if each cell of $\crLev$ is positively invariant for $\dif$.
Denote by $\uTriv$ the subgroup of $\StabilizerInv{\func,\crLev}$ consisting of $\crLev$-trivial diffeomorphisms.
Evidently, $\uTriv$ is normal in $\StabilizerInv{\func,\crLev}$.
Given $\gdif\in\StabilizerInv{\func,\crLev}$ we will write $\gdif \usim \dif$ whenever $\gdif^{-1}\circ\dif \in \uTriv$.

\medskip

Furthermore, let $\crNbh$ be an $\func$-regular neighbourhood of $\crLev$.
Then $\dif(\crNbh)=\crNbh$ and so $\dif$ yields a certain permutation of connected components of $\partial\crNbh$.
Say that $\dif$ is \emph{$\partial\crNbh$-trivial}, if each connected component of $\partial\crNbh$ if positively invariant for $\dif$.
Denote by $\dTriv$ the normal subgroup of $\StabilizerInv{\func,\crLev}$ consisting of $\partial\crNbh$-trivial diffeomorphisms.
Again for $\gdif\in\StabilizerInv{\func,\crLev}$ we write $\gdif \dsim \dif$ whenever $\gdif^{-1}\circ\dif\in\dTriv$.

\begin{lemma}\label{lm:uTriv_dTriv_prop}{\rm see~\cite[Theorems~6.2, 7.1]{Maksymenko:AGAG:2006}.}
Let $\gdif,\dif\in\StabilizerInv{\func,\crLev}$ and $\crNbh$ be an $\func$-regular neighbourhood of $\crLev$.
Then the following statements hold.
\begin{enumerate}
\item
Suppose there exists at least one edge $\delta$ being positively invariant for $\gdif^{-1}\circ\dif$.
Then all cells of $\crLev$ are also positively invariant for $\gdif^{-1}\circ\dif$, and in particular, $\gdif \usim \dif$.

\item
Let $\Wman$ be an open neighbourhood of $\crNbh$.
If $\gdif \usim \dif$, then there exists an isotopy of $\gdif$ in $\StabilizerInv{\func,\crLev}$ supported in $\Wman$ to some $\gdif'$ such that $\gdif'=\dif$ on $\crNbh$.

In particular, every $\crLev$-trivial diffeomorphism is isotopic in $\StabilizerInv{\func,\crLev}$ via an isotopy supported in $\Wman$ to a diffeomorphism fixed on $\Nman$.

\item
$\uTriv \subset \dTriv$, and so the relation $\gdif \usim \dif$ always implies $\gdif \dsim \dif$.

\item
Suppose that either 
\begin{itemize}
\item
$\gdif$ and $\dif$ are isotopic as diffeomorphisms of $\Mman$ or
\item
$\crNbh$ can be embedded into $\RRR^2$.
\end{itemize}
Then $\uTriv=\dTriv$ and so the relations $\gdif \dsim \dif$ and $\gdif \usim \dif$ are equivalent.
\end{enumerate}
\end{lemma}
\begin{proof}
Statement (1) is a consequence of~\cite[Claim~7.1.1]{Maksymenko:AGAG:2006}, (2) follows from~\cite[Theorem~6.2 \& Lemma~6.4]{Maksymenko:AGAG:2006} for Morse maps and from~\cite[Theorem~5]{Maksymenko:ProcIM:ENG:2010} for all $\func\in\FFF(\Mman,\Pman)$, (3) follows from (2), and (4) from~\cite[Theorem~7.1]{Maksymenko:AGAG:2006}.
\end{proof}

Now let $\bComp$ be a connected component of $\partial\crNbh$, and $\hbComp$ be a regular component of some level-set of $\func$ such that $\bComp$ and $\hbComp$ bound a cylinder $\bCyl$ containing no critical points of $\func$, see Figure~\ref{fig:crlev_ext_nbh}.
Denote also $\hcrNbh = \crNbh \cup \bCyl$.
\begin{figure}[ht]
\centerline{\includegraphics[width=7cm]{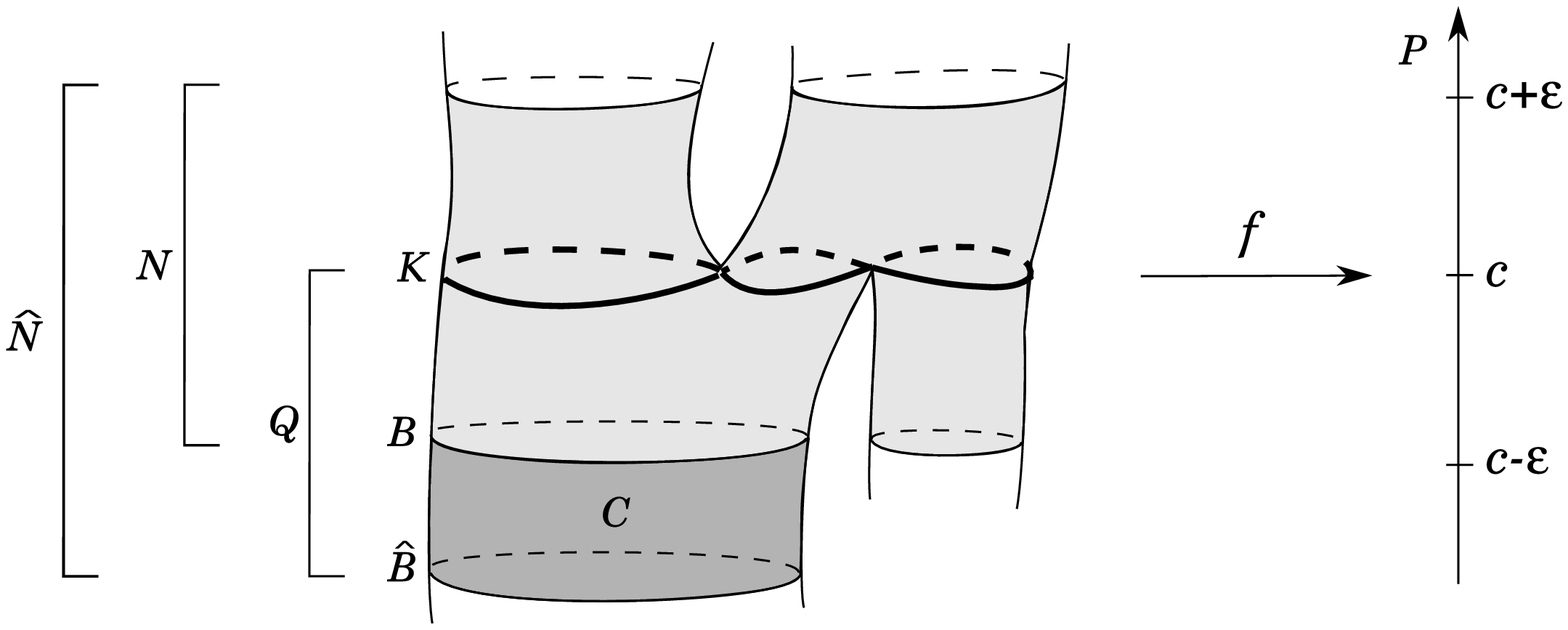}}
\caption{Extended neighbourhood of $\crLev$}
\protect\label{fig:crlev_ext_nbh}
\end{figure}

\begin{lemma}\label{lm:isot_rel_hcrNbh}{\rm see~\cite[Lemma~4.14]{Maksymenko:AGAG:2006}}.
Suppose $\crLev$ is not a local extreme of $\func$.
If $\gdif, \dif\in \Stabilizer{\func,\hbComp}$ are such that
\begin{itemize}
\item
$\gdif = \dif$ on some neighbourhood of $\hcrNbh$ and
\item
$\gdif$ and $\dif$ are isotopic in $\Stabilizer{\func}$,
\end{itemize}
then they are isotopic in $\Stabilizer{\func}$ relatively to some neighbourhood of $\hcrNbh$.
\end{lemma}
\begin{proof}
It suffices to prove this lemma for the case when $\dif = \id_{\Mman}$, and so we are in the situation when $\gdif$ belongs to $\fStabId$ and is also fixed on $\crNbh$.

Suppose $\Mman$ is orientable.
In this case one can define a smooth flow $\aFlow:\Mman\times\RRR\to\Mman$ such that $\gdif\in\fStabId$ if and only if there exists a $C^{\infty}$-function $\alpha_{\gdif}:\Mman\to\RRR$ satisfying the identity: $\gdif(x) = \aFlow(x,\alpha_{\gdif}(x))$ for all $x\in\Mman$, see~\cite[Theorem~3]{Maksymenko:ProcIM:ENG:2010}.
Moreover, this function is unique on any connected $\func$-adopted subsurface containing at least one critical point being \textit{not a non-degenerate local extreme of $\func$}.
By assumption $\hcrNbh$ contains such points and $\gdif$ is fixed on $\hcrNbh$.
Hence $\alpha_{\gdif}=0$ on $\hcrNbh$.
Then the isotopy between $\gdif$ and $\id_{\Mman}$ in $\fStabId$ can be given by the formula
$\gdif_t(x) = \aFlow(x,t\alpha_{\gdif}(x))$, $t\in[0,1]$, see~\cite[Lemma~4.14]{Maksymenko:AGAG:2006}.

If $\Mman$ is non-orientable, the proof follows by the arguments similar to the proof of~\cite[Theorem~3]{Maksymenko:ProcIM:ENG:2010} for non-orientable case.
\end{proof}
\begin{lemma}\label{lm:epi_SIsotId_f_hB__Z}
Suppose $\crLev$ is not a local extreme of $\func$.
Then there exists an epimorphism $\eta:\Stabilizer{\func,\hbComp} \longrightarrow \ZZZ$ having the following properties.
\begin{enumerate}
\item
$\uTriv = \eta^{-1}(m\ZZZ)$ for some $m\geq1$.
In particular, $\Stabilizer{\func,\hbComp} / \uTriv \cong \ZZZ_m$.

\item
Let $\Wman$ be any open neighbourhood of $\hcrNbh$ and $\gdif,\dif\in\Stabilizer{\func,\hbComp}$.
Then $\eta(\gdif)=\eta(\dif)$ if and only if there exists an isotopy of $\gdif$ in $\Stabilizer{\func,\hbComp}$ supported in $\Wman$ to some $\gdif'$ such that $\gdif'=\dif$ on $\hcrNbh$.
\end{enumerate}
\end{lemma}
\begin{proof}
Let $\Vman$ be the connected component of $\hcrNbh\setminus\crLev$ containing $\hbComp$ and $\dCyl = \overline{\Vman} \setminus \fSing$.
It is easy to see that $\dCyl$ is diffeomorphic to $S^1\times[0,1]\setminus F$, where $F$ is a finite subset of $S^1\times 1$, see Figures~\ref{fig:crlev_ext_nbh} and~\ref{fig:univ_cov_dCyl}.
\begin{figure}[ht]
\centerline{\includegraphics[width=12cm]{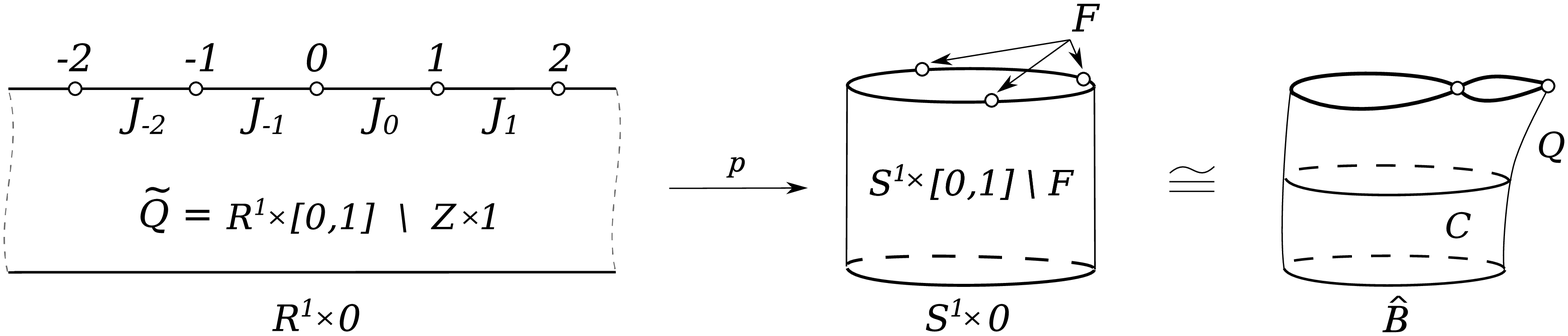}}
\caption{}
\protect\label{fig:univ_cov_dCyl}
\end{figure}
Let $p:\hdCyl\to\dCyl$ be the universal covering map for $\dCyl$.
Evidently, $\hdCyl$ is diffeomorphic with $\RRR\times[0,1] \setminus \ZZZ\times1$, see Figure~\ref{fig:univ_cov_dCyl}.
Let $b>0$ be the number of points in $F$.
It also equals to the number of connected components of $S^1\times 1\setminus F$.
Denote 
\[
J_i = (i,i+1) \times 1, \qquad i\in\ZZZ.
\]
Then 
\begin{equation}\label{equ:pJi_pJib}
p(J_i) = p(J_{i+b}), \ i\in\ZZZ.
\end{equation}
Now let $\dif\in \Stabilizer{\func,\hbComp}$.
Since $\dif$ is fixed on $\hbComp$, it follows that $\dif|_{\dCyl}$ lifts to a unique diffeomorphism $\hdif$ of $\hdCyl$ fixed on $\RRR^1\times 0$.
Then $\hdif$ ``shifts'' open intervals $\{J_i\}_{i\in\ZZZ}$ preserving their linear order.
In other words, there exists a unique $k\in\ZZZ$ such that $\hdif\bigl( J_i \bigr) \ = \ J_{i+k}$ for all $i\in\ZZZ$.
Define a map $\eta':\Stabilizer{\func,\hbComp} \longrightarrow \ZZZ$ by
\[
\eta'(\dif) = k.
\]
It is easy to check that $\eta'$ is in fact a homomorphism.

It follows from Eq.~\eqref{equ:pJi_pJib} that $\dif\in\uTriv$ if and only if $\eta'(\dif)$ is divided by $b$.
In other words
\begin{equation}\label{equ:uTriv_inv_bZ}
\uTriv = (\eta')^{-1}(b\ZZZ).
\end{equation}

Let us show that {\em $\eta'$ is a non-trivial homomorphism.}
Recall that there exists a Dehn twist $\tau \in \Stabilizer{\func,\hbComp}$ supported in $\bCyl$, see~\cite[\S~6]{Maksymenko:AGAG:2006}.
Then it is easy to see that $\eta'(\tau) = b$ or $-b$.
In particular, $\tau\in\uTriv$.

Hence the image of $\eta'$ is also a non-zero subgroup of $\ZZZ$, so $\eta'(\Stabilizer{\func,\hbComp}) = n\ZZZ$ for some $n\geq1$.
In particular, due to Eq.~\eqref{equ:uTriv_inv_bZ} $n$ must divide $b$.
Therefore the map $\eta:\Stabilizer{\func,\hbComp} \longrightarrow \ZZZ$ defined by
\[
\eta(\dif) = \eta'(\dif)/n
\]
is an epimorphism.

Property (1) for $\eta$ now follows from Eq.~\eqref{equ:uTriv_inv_bZ} with $m = b/n$.
It remains to check (2).

Let $\gdif,\dif\in \Stabilizer{\func,\hbComp}$ and $\hgdif,\hdif:\hdCyl\to\hdCyl$ be unique liftings of $\gdif|_{\dCyl}$ and $\dif|_{\dCyl}$ respectively fixed on $\RRR^1\times 0$.

Suppose there exists an isotopy $\{\gdif_t\}_{t\in[0,1]}$ in $\Stabilizer{\func,\hbComp}$ such that $\gdif_0 = \gdif$ and $\gdif_1 =\dif$ on $\hcrNbh$.
We claim that $\eta(\gdif) =\eta(\dif)$.

Indeed, let $\hgdif_t$ be the lifting of $\gdif_t|_{\dCyl}$ fixed on $\RRR^1\times 0$.
Then $\{\hgdif_t\}_{t\in[0,1]}$ is an isotopy between $\hgdif= \hgdif_0$ and $\hdif = \hgdif_1$.
Hence all $\hgdif_t$ shift boundary components $\{(i,i+1)\times 1\}_{i\in\ZZZ}$ of $\hdCyl$ in the same way, and so 
\[\eta(\gdif)= \eta(\gdif_0) = \eta(\gdif_t) = \eta(\gdif_1) =\eta(\dif).\]

Conversely, suppose $\eta(\gdif) =\eta(\dif)$.
Then $\gdif$ and $\dif$ interchange edges of $\crLev$ in the same way, and so by (1) of Lemma~\ref{lm:uTriv_dTriv_prop} $\gdif\usim \dif$.
Moreover, by (2) of that lemma $\gdif$ is isotopic to a diffeomorphism $\gdif'$ such that $\gdif' = \dif$ on $\crNbh$.
Hence $\gdif'\circ\dif^{-1}|_{\hcrNbh}$ is supported in a cylinder $\bCyl$, see Figures~\ref{fig:crlev_ext_nbh} and~\ref{fig:univ_cov_dCyl}, and so it is isotopic relatively to $\partial\bCyl$ to some degree $a$ of the Dehn twist $\tau$ mentioned above.
Therefore $\eta(\gdif'\circ\dif^{-1}) = \eta(\tau^a) = a k/n$.
However \[\eta(\gdif'\circ\dif^{-1}) = \eta(\gdif') - \eta(\dif) = \eta(\gdif) - \eta(\dif) = 0,\]
whence $a=0$.
This means that $\gdif'\circ\dif^{-1}|_{\bCyl}$ is isotopic to $\tau^0|_{\bCyl} = \id_{\bCyl}$ relatively to $\partial\bCyl$.
Hence by~\cite[Lemma~4.12(3)]{Maksymenko:AGAG:2006} that isotopy can be made $\func$-preserving.
Thus $\gdif'$ (and therefore $\gdif$) is isotopic in $\Stabilizer{\func,\hbComp}$ to some $\gdif''$ such that $\gdif''=\dif$ on $\hcrNbh$.
Lemma is completed.
\end{proof}

\section{Functions on $2$-disks and cylinders}\label{sect:funcs_2D_Cyl}
In this section we assume that $\Mman$ is either a $2$-disk or a cylinder, $\func\in\FFF(\Mman,\Pman)$, and $\hbComp$ is a connected component of $\partial\Mman$.
Our aim is to establish the following key result which will be proved in~\S\ref{sect:proof_prop_pr:pi0StabfB_in_ExtZ}.
\begin{proposition}\label{pr:pi0StabfB_in_ExtZ}
The group $\pi_0\StabilizerIsotId{\func,\partial\Mman}$ belongs to class $\classP$.
\end{proposition}

For the proof we need some preliminary statements.
\begin{lemma}\label{lm:StabIsotId_Stab}
\begin{enumerate}
\item[\rm(1)]
$\StabilizerIsotId{\func,\hbComp} = \Stabilizer{\func,\hbComp}$,
\item[\rm(2)]
$\pi_0 \StabilizerIsotId{\func,\hbComp} = \pi_0\StabilizerIsotId{\func,\partial\Mman}$.
\end{enumerate}
\end{lemma}
\begin{proof}
(1) Recall that by definition $\StabilizerIsotId{\func,\hbComp} := \Stabilizer{\func,\hbComp} \cap  \DiffId(\Mman,\hbComp)$.
Therefore we should only prove that each $\dif\in\Stabilizer{\func,\hbComp}$ is isotopic to $\id_{\Mman}$ relatively to $\partial\Mman$.
By 5) of Theorem~\ref{th:right_action_props} $\dif$ is isotopic in $\Stabilizer{\func,\hbComp}$ to a diffeomorphism $\dif'$ fixed on some neighbourhood of $\hbComp$.
Since $\Mman$ is either a $2$-disk or a cylinder, it follows from~\cite{Smale:ProcAMS:1959, Gramain:ASENS:1973} that then $\dif'$ is isotopic to $\id_{\Mman}$ relatively to $\hbComp$.

(2) If $\Mman=\Disk$, then $\hbComp = \partial\Mman$ and the statement is trivial.
Suppose $\Mman=S^1\times I$.
Then by 1) and 4) of Theorem~\ref{th:right_action_props} we have the following isomorphisms:
\[
\pi_0 \StabilizerIsotId{\func,\hbComp} \ \cong \
\pi_1 \OrbitPathComp{\func,\hbComp} \ \cong \
\pi_1 \OrbitPathComp{\func,\partial\Mman} \ \cong \
\pi_0 \StabilizerIsotId{\func,\partial\Mman}.
\]
Lemma is completed.
\end{proof}

Thus due to (2) for the proof of Proposition~\ref{pr:pi0StabfB_in_ExtZ} it suffices to show that $\pi_0 \StabilizerIsotId{\func,\hbComp} \in\classP$.
Of course, this replacement is non-trivial only for $\Mman=S^1\times I$.

Let $\Zman$ be the union of all critical components of all level sets of $\func$, $\Uman$ be a connected component of $\Mman\setminus\Zman$ containing $\hbComp$, and $\crLev$ be that unique critical component from $\Zman$ which intersects $\overline{\Uman}$.
Roughly speaking, $\crLev$ is the ``closest'' to $\hbComp$ critical component of some level set of $\func$.

Let also $\crNbh$ be an $\func$-regular neighbourhood of $\crLev$ that does not contain $\hbComp$ and
\[
\hcrNbh = \crNbh \cup \Uman,
\qquad
\bCyl = \overline{\Uman \setminus\Nman}.
\]
Then we are in the notations and under assumptions of Lemma~\ref{lm:epi_SIsotId_f_hB__Z} for a special case when $\hbComp$ is a boundary component of $\partial\Mman$.

By (1) or Lemma~\ref{lm:StabIsotId_Stab} each $\dif\in\Stabilizer{\func,\hbComp}$ is isotopic to $\id_{\Mman}$, whence by (4) of Lemma~\ref{lm:uTriv_dTriv_prop} we get that $\uTriv=\dTriv$.
Moreover, by Lemma~\ref{lm:isot_rel_hcrNbh} there exists an epimorphism
\[ \eta: \Stabilizer{\func,\hbComp}  \longrightarrow \ZZZ\]
satisfying $\uTriv = \eta^{-1}(m\ZZZ)$ for some $m\geq1$, and so
\[
\Stabilizer{\func,\hbComp} / \uTriv \ = \ \Stabilizer{\func,\hbComp} / \dTriv \ \cong \ \ZZZ_m.
\]

\begin{lemma}\label{lm:nSurf_g_orbit}
Let $\gdif\in\Stabilizer{\func,\hbComp}$ be such that $\gdif(\nSurf)\cap\nSurf = \varnothing$ for some connected component $\nSurf$ of $\overline{\Mman\setminus\crNbh}$.
If $\eta(\gdif)=1$, then $\gdif^i(\nSurf)\cap\nSurf = \varnothing$ for $i=1,\ldots,m-1$, and $\gdif^{m}(\nSurf)=\nSurf$.
\end{lemma}
\begin{proof}
Notice that under assumption of lemma $\gdif \not\in\dTriv$, whence $m>1$.
Moreover, $\eta(\gdif^m) = m \in m\ZZZ$, and so $\gdif^m\in\dTriv$.
Therefore $\gdif^m(\nSurf) = \nSurf$.

It remains to consider the case $i\in\{1,\ldots,m-1\}$.
Let $\hMman$ be a closed surface obtained by gluing every connected component of $\partial\Mman$ with a $2$-disk.
Since $\Mman$ is either a $2$-disk or a cylinder, we obtain that $\hMman$ is a $2$-sphere.
Then $\hMman\setminus\crLev$ if a union of open $2$-disks, and so we have a cellular subdivision of $\hMman$ by $0$- and $1$-cells of $\crLev$ and $2$-cells being connected components of $\hMman\setminus\crLev$.

Let $\dif\in\Stabilizer{\func,\hcrNbh}$.
Since $\dif$ leaves invariant boundary components of $\partial\Mman$, it extends to a certain homeomorphism $\hdif$ of all of $\hMman$ which preserves orientation of $\hMman$, and therefore is also homotopic to $\id_{\hMman}$.
Then by~\cite[Proposition~5.4]{Maksymenko:MFAT:2010} either
\begin{itemize}
\item[(a)]
all cells are positively invariant for $\hdif$, or
\item[(b)]
the number of positively invariant cells of $\hdif$ is equal to the Euler characteristic of $\hMman$, i.e. to $2$.
\end{itemize}
In particular this holds for $\dif = \gdif^i$, $i=1,\ldots,m-1$.

Now let $\delta_0,\delta_1$ be positively invariant cells of $\hgdif$.
Then they are also positively invariant for $\hgdif^i$, $i=1,\ldots,m-1$.
Notice also that these cells do not intersect $\nSurf$, since $\gdif(\nSurf) \cap \nSurf = \varnothing$.

Therefore if we assume that $\gdif^{i}(\nSurf) = \nSurf$ for some $i=1,\ldots,m-1$, then $\hgdif^i$ would have at least $3$ positively invariant cells of $\hMman$, and by (a) all other cells of $\hMman$ must also be $\hgdif^i$-invariant.
But this would mean that $\gdif^{i} \in \dTriv$ which is possible only if $i$ is a multiple of $m$.
We get a contradiction with the assumption $i\in\{1,\ldots,m-1\}$.
Hence $\gdif^{i}(\nSurf) \cap \nSurf = \varnothing$ for $1 \leq i \leq m-1$.
\end{proof}

Fix any $\gdif\in\Stabilizer{\func,\hbComp}$ with $\eta(\gdif)=1$ and let 
\begin{equation}\label{equ:g_inv_comp}
\bCyl = \iSurf_0, \ \iSurf_1, \ \ldots, \ \iSurf_a
\end{equation}
be all $\gdif$-invariant connected components of $\overline{\Mman\setminus\crNbh}$, 
\[
\iSurf = \iSurf_1 \cup \cdots \cup \iSurf_a
\]
be the union of all these components except for $\bCyl$, and
\begin{align*}
\hiSurf_i &= \iSurf_i\cap\partial\crNbh, &
\hiSurf &= \iSurf\cap\partial\crNbh,
\end{align*}
see Figure~\ref{fig:g_inv_noninv_comps}.
By (1) of Lemma~\ref{lm:epi_SIsotId_f_hB__Z} these notation does not depend on a particular choice of such $\gdif$.
\begin{figure}[ht]
\centerline{\includegraphics[width=9cm]{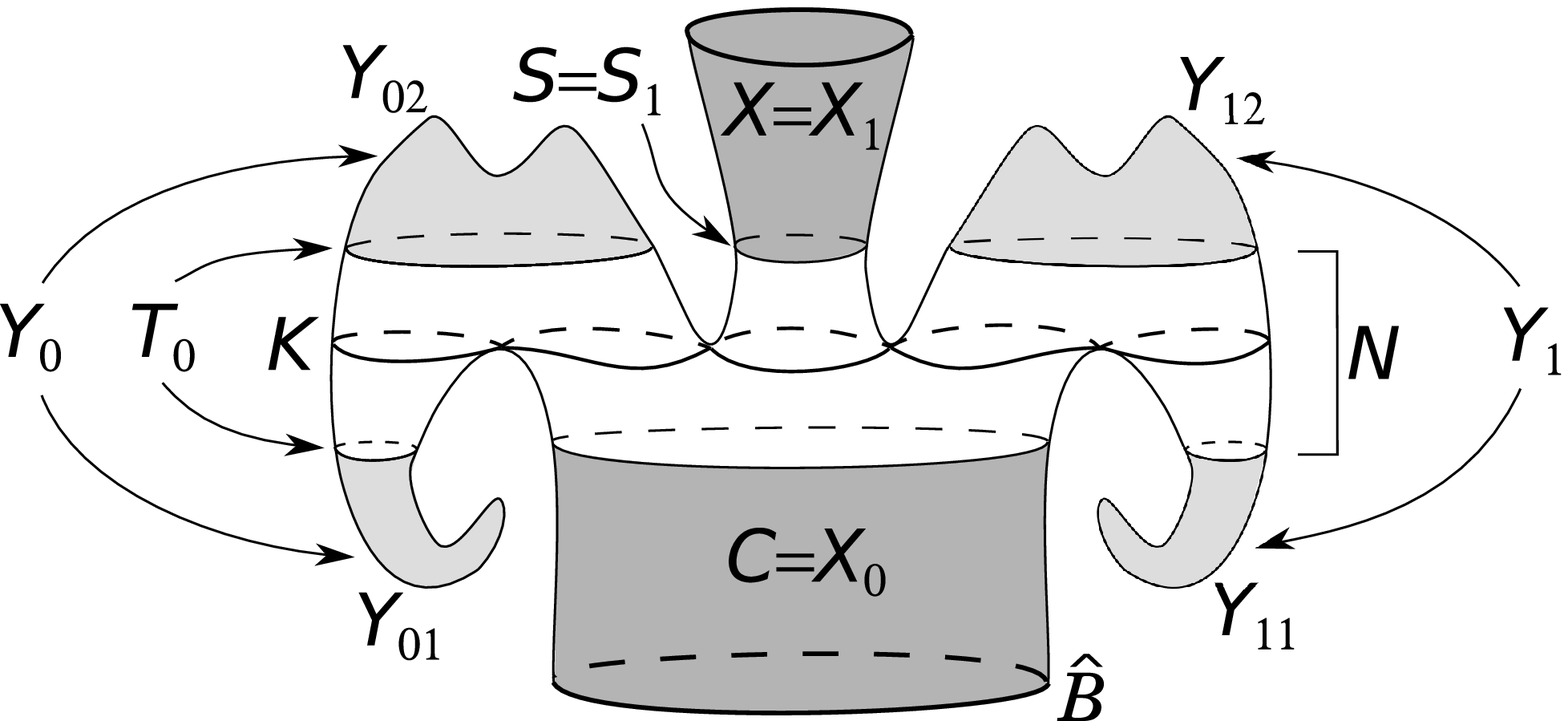}}
\caption{}
\protect\label{fig:g_inv_noninv_comps}
\end{figure}

\begin{lemma}\label{lm:g_with_eta_g_1}
There exists $\gdif\in\Stabilizer{\func,\hbComp}$ fixed near $\iSurf$ and satisfying $\eta(\gdif)=1$.
\end{lemma}
\begin{proof}
Let $\dif\in\Stabilizer{\func,\hbComp}$ be any element with $\eta(\dif)=1$.
Then $\dif$ leaves invariant every connected component of $\hiSurf$, and preserves their orientation.
Therefore $\dif$ is isotopic in $\Stabilizer{\func,\hbComp}$ to a diffeomorphism $\dif'$ fixed on some neighbourhood of $\hiSurf$.
Now change $\dif'$ on $\iSurf$ by the identity and denote the obtained diffeomorphism by $\gdif$.
Then $\gdif\in\Stabilizer{\func,\hbComp}$ and $\eta(\gdif)=1$.
\end{proof}

Let $\gdif\in\Stabilizer{\func,\hbComp}$ be such that $\eta(\gdif)=1$.
It follows from Lemma~\ref{lm:nSurf_g_orbit} that connected components of $\overline{\Mman\setminus\crNbh}$ that are \emph{not $\gdif$-invariant} can be enumerated as follows:
\begin{equation}\label{equ:non_fixed_vertices}
\begin{matrix}
 \nSurf_{0,1} & \nSurf_{0,2}& \cdots & \nSurf_{0,b} \\
 \nSurf_{1,1} & \nSurf_{1,2}& \cdots & \nSurf_{1,b} \\
 \cdots    & \cdots   & \cdots & \cdots \\
 \nSurf_{m-1,1} & \nSurf_{m-1,2} & \cdots & \nSurf_{m-1,b}
\end{matrix}
\end{equation}
so that 
\[\gdif(\nSurf_{j,\qind}) = \gdif(\nSurf_{j+1 \ \mathrm{mod} \ m, \ \qind})\]
for all $j,\qind$.
In other words, $\gdif$ cyclically shifts down the rows of Eq.~\eqref{equ:non_fixed_vertices}, see Figure~\ref{fig:g_inv_noninv_comps}.
Denote
\[ 
\nSurf_j =  \nSurf_{j,1} \cup \nSurf_{j,2} \cup \cdots \cup \nSurf_{j,b}, 
\qquad\qquad \nSurf = \bigcup_{j=0}^{m-1} \nSurf_j,
\]
\[
\hnSurf_{j,\qind} = \partial\nSurf_{j,\qind} \cap \crNbh,
\qquad\qquad
\hnSurf_{j} = \partial\nSurf_{j} \cap \crNbh.
\]
Then
\begin{align*}
\gdif^{j}(\nSurf_0) &= \nSurf_{j},
&
\nSurf_{j} \cap \nSurf_{j'}  &= \varnothing
\end{align*}
for $j\not=j' = 0,\ldots,m-1$.
Consider also the restrictions
\begin{align*}
\func_{\iSurf} &= \func|_{\iSurf}: \iSurf\to\Pman,
&
\func_{\nSurf_j} &= \func|_{\nSurf_j}: \nSurf_j \to\Pman.
\end{align*}

\begin{lemma}\label{lm:pi0StabfB_sructure}
In the notation above there exists an isomorphism
\begin{equation}\label{equ:main_isomorphism}
\isoStabfB: \pi_0\Stabilizer{\func,\hbComp} 
\longrightarrow
\pi_0\Stabilizer{\func_{\iSurf},\hiSurf} \times \bigl( \pi_0\Stabilizer{\func_{\nSurf_0}, \hnSurf_0} \wrm{m} \ZZZ\bigr).
\end{equation}
For $m=1$, $\psi$ reduces to an isomorphism
\[
\isoStabfB: \pi_0\Stabilizer{\func,\hbComp} 
\longrightarrow \pi_0\Stabilizer{\func_{\iSurf},\hiSurf} \times  \ZZZ.
\]
\end{lemma}
\begin{proof}
Choose $\gdif\in\Stabilizer{\func,\hbComp}$ fixed near $\iSurf$ and satisfying $\eta(\gdif)=1$, see Lemma~\ref{lm:g_with_eta_g_1}.

Let $\gamma\in \pi_0\Stabilizer{\func,\hbComp}$.
By (1) of Lemma~\ref{lm:epi_SIsotId_f_hB__Z} we can take a representative $\dif\in\gamma$ such that $\gdif^{-\eta(\dif)}\circ\dif$ is fixed on some neighbourhood of $\hcrNbh$.
As $\gdif$ is fixed near $\iSurf$, we have that 
\begin{align*}
\gdif^{-\eta(\dif)}\circ\dif(\iSurf) = \dif(\iSurf) &=\iSurf, &
\gdif^{-\eta(\dif)} \circ \dif(\nSurf_j) &= \nSurf_j,
\end{align*}
for all $j$, whence
\begin{align*}
\dif|_{\iSurf} &\,\in\, \Stabilizer{\func_{\iSurf},\hiSurf},
&
\gdif^{-j-\eta(\dif)} \circ \dif \circ \gdif^{j}|_{\nSurf_0}
&\,\in\, \Stabilizer{\func_{\nSurf_0},\hnSurf_0}
\end{align*}
for $j=1,\ldots,m$.
Therefore we obtain a function 
\[
\sigma_{\dif}: \ZZZ_m\longrightarrow \pi_0\Stabilizer{\func_{\nSurf_0},\hnSurf_0},
\qquad
\sigma(j) = \bigl[ \gdif^{-j-\eta(\dif)} \circ \dif \circ \gdif^{j}|_{\nSurf_0} \bigr]
\]
$j=0,\ldots,m-1$.

Consider the following element belonging to $\pi_0\Stabilizer{\func_{\iSurf},\hiSurf} \times \bigl( \pi_0\Stabilizer{\func_{\nSurf_0}, \hnSurf_0} \wrm{m} \ZZZ\bigr)$:
\[
\isoStabfB(\gamma) = \Bigl( [\, \dif|_{\iSurf}\,]\, , \, \sigma_{\dif}\, , \, \eta(\dif) \Bigr).
\]
We claim that the correspondence $\gamma\longmapsto\isoStabfB(\gamma)$ is the desired isomorphism~\eqref{equ:main_isomorphism}.

\medskip

{\bf Step 1.}
First we show that {\em $\isoStabfB(\gamma)$ does not depend on a particular choice of a representative $\dif\in\gamma$ such that $\gdif^{-\eta(\dif)}\circ\dif$ is fixed on some neighbourhood of $\hcrNbh$.}

Indeed, let $\dif' \in \gamma$ be another element such that $\gdif^{-\eta(\dif')}\circ\dif'$ is fixed near $\hcrNbh$.
Then $\dif' = \dif = \gdif^{\eta(\dif)}$ near $\hcrNbh$ and $\dif'$ is isotopic to $\dif$ in $\StabilizerId{\func,\hbComp}$.

In particular, it follows from (1) of Lemma~\ref{lm:epi_SIsotId_f_hB__Z} that $\eta(\dif)=\eta(\dif')$.

Moreover, by Lemma~\ref{lm:isot_rel_hcrNbh} $\dif$ and $\dif'$ are isotopic in $\StabilizerId{\func,\hbComp}$ relatively some neighbourhood of $\hcrNbh$.
This implies that $\dif|_{\iSurf}$ is isotopic to $\dif'|_{\iSurf}$ relatively some neighbourhood of $\hiSurf$, and for each $j=0,\ldots,m-1$ the restriction $\gdif^{-j-\eta(\dif)} \circ \dif \circ \gdif^{j}|_{\nSurf_0}$ is isotopic to $\gdif^{-j-\eta(\dif)} \circ \dif' \circ \gdif^{j}|_{\nSurf_0}$ relatively some neighbourhood of $\hnSurf_0$.
In other words, 
\[ [\, \dif|_{\iSurf}\,]=[\, \dif'|_{\iSurf}\,] \in \pi_0\StabilizerIsotId{\func_{\iSurf}, \hiSurf}, \]
\[ [\, \gdif^{-j-\eta(\dif)} \circ \dif \circ \gdif^{j}|_{\nSurf_0}\,]=[\, \gdif^{-j-\eta(\dif)} \circ \dif' \circ \gdif^{j}|_{\nSurf_0} \,] \in \pi_0\Stabilizer{\func_{\nSurf_0}, \hnSurf_0}, \]
$j=0,\ldots,m-1$. Hence $\isoStabfB(\gamma)$ does not depend on a particular choice of such $\dif$.

\medskip

{\bf Step 2.}
{\em $\isoStabfB$ is a \textit{homomorphism}.}
Let $\dif_0,\dif_1\in\Stabilizer{\func,\hbComp}$.
We have to show that 
\[
\isoStabfB([\dif_0\circ\dif_1]) = \isoStabfB([\dif_0]) \cdot \isoStabfB([\dif_1]).
\]
Put $k_i = \eta(\dif_i)$, $i=0,1$.
Since $\eta$ is a homomorphism, $\eta(\dif_0\circ\dif_1) = k_0+k_1$.

By Step 1 we can assume that $\gdif^{-k_i}\circ\dif_i$ is fixed on $\hcrNbh$, $i=0,1$.
Define the following four functions 
\[ \sigma_0, \sigma_1,\sigma, \bar{\sigma}: \ZZZ_m\longrightarrow \pi_0\Stabilizer{\func_{\nSurf_0},\hnSurf_0} \] 
by
\begin{align*}
\sigma_0(j) &= \bigl[ \gdif^{-j-k_0} \circ \dif_i \circ \gdif^{j}|_{\nSurf_0} \bigr], &
\sigma_1(j) &= \bigl[ \gdif^{-j-k_1} \circ \dif_i \circ \gdif^{j}|_{\nSurf_0} \bigr], \\
\sigma(j) &= \bigl[ \gdif^{-j-k_0-k_1} \circ \dif_0\circ\dif_1 \circ \gdif^{j}|_{\nSurf_0} \bigr], &
\bar{\sigma}(j) &= \sigma_0(j+k_1)\circ\sigma_1(j)
\end{align*}
for $j=0,\ldots,m-1$.
Then
\begin{align*}
\isoStabfB([\dif_i]) &= \bigl([\,\dif_i|_{\iSurf}\,], \ \sigma_i, \  k_i \bigr), \qquad i=0,1, \\
\isoStabfB([\dif_0\circ\dif_1]) &= \bigl(\,[\dif_0\circ\dif_1|_{\iSurf}\,], \ \sigma, \  k_0+k_1 \bigr),
\end{align*}
and by the definition of multiplication 
\begin{align*}
\isoStabfB([\dif_0]) \circ \isoStabfB([\dif_1]) &=
\Bigl([\,\dif_0|_{\iSurf}\,], \ \sigma_0, \  k_0 \Bigr) \Bigl([\,\dif_1|_{\iSurf}\,], \ \sigma_1, \ k_1 \Bigr) \\
&= \Bigl([\,\dif_0|_{\iSurf}\,] \circ [\,\dif_1|_{\iSurf}\,], \ \bar\sigma, \  k_0+k_1 \Bigr)=
\Bigl([\,\dif_0 \circ \dif_1|_{\iSurf}\,], \ \bar\sigma, \  k_0+k_1 \Bigr).
\end{align*}
It remains to show that $\bar\sigma = \sigma$.
Let $j=0,\ldots,m-1$, then
\begin{align*}
\sigma(j) &= \bigl[ \gdif^{-j-k_0-k_1} \circ \dif_0\circ\dif_1 \circ \gdif^{j}|_{\nSurf_0} \bigr] \\
&= \bigl[\gdif^{-(j+k_1)-k_0}\circ\dif_0\circ \gdif^{j+k_1}|_{\nSurf_0}\bigr] \circ \bigl[ \gdif^{-j-k_1} \circ \dif_1 \circ \gdif^{j}|_{\nSurf_0}\bigr] \\
&= \sigma_0(j+k_1) \circ \sigma_1(j) = \bar{\sigma}(j).
\end{align*}
Thus $\isoStabfB$ is a homomorphism.

\medskip

{\bf Step 3.} 
{\em $\isoStabfB$ is a \textit{monomorphism}.}
Let $\dif\in\Stabilizer{\func,\hbComp}$ be such that $\gdif^{-\eta(\dif)}\circ\dif$ is fixed near $\hcrNbh$, and suppose that $[\dif]\in\ker(\isoStabfB)$.
This means that
\begin{align*}
 [ \dif|_{\iSurf} ] &= [\id_{\iSurf}] \in \pi_0\Stabilizer{\func_{\iSurf}, \hiSurf}, \\
 [ \gdif^{-j}\circ\dif\circ\gdif^{j}|_{\nSurf_0} ] &= [ \id_{\nSurf_0}] \in\pi_0\Stabilizer{\func_{\nSurf_0},\hnSurf_0}, \\
 \eta(\dif) &=0, 
\end{align*}
for $j=0,\ldots,m-1$. 
In other words, $\dif|_{\iSurf}$ is isotopic in $\StabilizerId{\func_{\iSurf},\hiSurf}$ to $\id_{\iSurf}$, and $\dif|_{\nSurf_j}$ is isotopic in $\StabilizerId{\func_{\nSurf_j},\hnSurf_j}$ to $\id_{\nSurf_j}$.
These isotopies give an isotopy between $\dif$ and $\id_{\Mman}$ in $\Stabilizer{\func,\hbComp}$.
Hence $[\dif]=[\id_{\Mman}]\in\Stabilizer{\func,\hbComp}$, and so $\ker(\isoStabfB)$ is trivial.

\medskip

{\bf Step 4.} 
{\em $\isoStabfB$ is \textit{surjective}.}
Let $\hdif\in\Stabilizer{\func_{\iSurf},\hiSurf}$, $\sigma:\ZZZ_m\to\pi_0\Stabilizer{\func_{\nSurf_0},\hnSurf_0}$, and $k\in\ZZZ$.
We have to find $\dif\in\Stabilizer{\func,\hbComp}$ with $\isoStabfB([\dif]) = ([\hdif], \sigma, k)$.
For each $j\in\ZZZ_m$ choose $\dif_j\in\Stabilizer{\func_{\nSurf_0},\hnSurf_0}$ such that $\sigma(j) = [\dif_j]$.
Due to 5) of Theorem~\ref{th:right_action_props} we can assume that $\hdif$ is fixed near $\hiSurf$ and each $\dif_j$ is fixed near $\hnSurf_0$.
Define $\dif$ by the formula:
\[
\dif(x) = 
\begin{cases}
\gdif^{k}(x), & x\in \crNbh, \\
\gdif^{k}\circ\hdif(x), & x\in \iSurf, \\
\gdif^{j+k}\circ\dif_j \circ \gdif^{-j}(x), & x\in \nSurf_j, \ j=0,\ldots,m-1.
\end{cases}
\]
Then it is easy to check that $\dif\in\Stabilizer{\func,\hbComp}$ and $\isoStabfB([\dif]) =([\hdif], \sigma, k)$.
Lemma~\ref{lm:pi0StabfB_sructure} is completed.
\end{proof}

\begin{lemma}\label{lm:pi0StabIsotId_simple_cases}
{\rm 1)}
Let $\func\in\FFF(S^1\times I,\Pman)$ be a map without critical points.
Then 
\begin{align*}
\pi_0 \StabilizerIsotId{\func,S^1\times 0} = \pi_0 \StabilizerIsotId{\func,\partial(S^1\times I)} = 0.
\end{align*}

{\rm 2)}
Let $\func\in\FFF(\Disk,\Pman)$ be a map having exactly one critical point, which therefore must be a local extreme.
\begin{itemize}
\item[\rmfamily (a)]
If $z$ is a {\bfseries non-degenerate} local extreme of $\func$, then $\pi_0\StabilizerIsotId{\func,\partial\Disk} = 0.$
\item[\rmfamily (b)]
Suppose $z$ is a {\bfseries degenerate} local extreme of $\func$.
Then $\pi_0 \StabilizerIsotId{\func,\partial\Disk} \cong \ZZZ$.
\end{itemize}
\end{lemma}
\begin{proof}
These statements are contained in the previous papers by the author, though they were not explicitly formulated.
In fact, statement 1) follows from~\cite[Lemma~4.12(2,3)]{Maksymenko:AGAG:2006}, statement 2(a) from~\cite[Eq~(25)]{Maksymenko:TA:2003} or from results of~\cite{Maksymenko:CEJM:2009, Maksymenko:hamv2}, and statement 2(b) from results of~\cite{Maksymenko:MFAT:2009}.
We leave the details to the reader.
\end{proof}

\subsection{Proof of Proposition~\ref{pr:pi0StabfB_in_ExtZ}}\label{sect:proof_prop_pr:pi0StabfB_in_ExtZ}
Due to (2) of Lemma~\ref{lm:StabIsotId_Stab} it suffices to prove that $\pi_0\StabilizerIsotId{\func,\hbComp} \in \classP$.

If $\crLev$ is either empty or consists of a unique point, then by Lemma~\ref{lm:pi0StabIsotId_simple_cases} $\pi_0\StabilizerIsotId{\func,\hbComp}$ is either trivial or isomorphic with $\ZZZ$.
Therefore it belongs to the class $\classP$.

Suppose now that $\crLev$ consists of more than one point, and let $n$ be the total number of critical points of $\func$ in all of $\Mman$.
We will use induction on $n$.

If $n=0$, then we are in the case 1) of Lemma~\ref{lm:pi0StabIsotId_simple_cases} which is already considered.
Suppose Proposition~\ref{pr:pi0StabfB_in_ExtZ} is proved for all $n<k$ for some $k\geq1$.
Let us establish it for $n=k$.

Preserving notation of Lemma~\ref{lm:pi0StabfB_sructure} let $\crLev$ be the ``closest'' to $\hbComp$ critical component of some level set of $\func$, see beginning of \S\ref{sect:funcs_2D_Cyl}.
Since $\iSurf$ is a disjoint union of surfaces $\iSurf_i$, $i=1,\ldots,a$, as well as $\nSurf_0$ is a disjoint union of $\nSurf_{0,\qind}$, $\qind=1,\ldots,b$, it follows that 
\begin{align*}
  \pi_0\Stabilizer{\func_{\iSurf},\hiSurf} &\cong 
    \mathop{\times}\limits_{i=1}^{a} \pi_0\Stabilizer{\func_{\iSurf_i},\hiSurf_i},
&
  \pi_0\StabilizerIsotId{\func_{\nSurf_0},\hiSurf_0} &\cong
     \mathop{\times}\limits_{\qind=1}^{b} \pi_0\Stabilizer{\func_{\nSurf_{0,\qind}},\hnSurf_{0,\qind}},
\end{align*}
whence from Lemma~\ref{lm:pi0StabfB_sructure} we get an isomorphism
\[
\pi_0\Stabilizer{\func,\hcrNbh} \ \cong \
\left(\mathop{\times}\limits_{i=1}^{a} \pi_0\Stabilizer{\func_{\iSurf_i},\hiSurf_i}\right)
\times 
\left( \left(
\mathop{\times}\limits_{\qind=1}^{b} \pi_0\Stabilizer{\func_{\nSurf_{0,\qind}},\hnSurf_{0,\qind}} 
\right) \wrm{m} \ZZZ \right).
\]
As each pair $(\Mman,\hbComp)$, $(\iSurf_i, \hiSurf_i)$, and $(\nSurf_{j,\qind}, \hnSurf_{j,\qind})$ is diffeomorphic either with $(\Disk,\partial\Disk)$ or with $(S^1\times I, S^1\times0)$, it follows from (1) of Lemma~\ref{lm:StabIsotId_Stab} that
\begin{align*}
\StabilizerIsotId{\func,\hcrNbh} &= \Stabilizer{\func,\hcrNbh}, 
&
\StabilizerIsotId{\func_{\iSurf_i},\hiSurf_i} &= \Stabilizer{\func_{\iSurf_i},\hiSurf_i},
&
\StabilizerIsotId{\func_{\nSurf_{0,\qind}},\hnSurf_{0,\qind}} &= \Stabilizer{\func_{\nSurf_{0,\qind}},\hnSurf_{0,\qind}},
\end{align*}
so we also have an isomorphism
\[
\pi_0\StabilizerIsotId{\func,\hcrNbh} \ \cong \
\left(\mathop{\times}\limits_{i=1}^{a} \pi_0\StabilizerIsotId{\func_{\iSurf_i},\hiSurf_i}\right)
\times 
\left( \left(
\mathop{\times}\limits_{\qind=1}^{b} \pi_0\StabilizerIsotId{\func_{\nSurf_{0,\qind}},\hnSurf_{0,\qind}} 
\right) \wrm{m} \ZZZ \right).
\]
Notice that each of the restrictions $\func|_{\iSurf_i}$ and $\func|_{\nSurf_{0,\qind}}$ has less critical points than $n$, whence by inductive assumption $\pi_0\StabilizerIsotId{\func_{\iSurf_i},\hiSurf_i}$ and $\pi_0\StabilizerIsotId{\func_{\nSurf_{0,\qind}},\hnSurf_{0,\qind}}$ belong to the class $\classP$.
Hence $\pi_0\StabilizerIsotId{\func,\hcrNbh}\in\classP$ as well.
Proposition~\ref{pr:pi0StabfB_in_ExtZ} is completed.

\section{Proof of Theorem~\ref{th:structure_of_pi1_fOrb}}\label{sect:proof:th:structure_of_pi1_fOrb}
Let $\Mman$ be a compact orientable surface distinct from $\Sphere$ and $\Torus$.
Then $\DiffId(\Mman,\partial\Mman)$ is contractible, and by 1) and 4) of Theorem~\ref{th:right_action_props} for each  $\func\in\FFF(\Mman,\Pman)$ we have the following isomorphisms
\[
\pi_1\fOrb  \ \cong \ \pi_1\Orbit{\func, \partial\Mman} \ \cong \ \pi_0\StabilizerIsotId{\func,\partial\Mman}.
\]

Therefore it suffices to prove that the class $\classP$ coincides with each of the following classes of groups:
\begin{align*}
&\{ \ \pi_0\StabilizerIsotId{\func,\partial\Mman} \ \mid \ \func\in\Morse(\Mman,\Pman) \ \}, & 
&\{ \ \pi_0\StabilizerIsotId{\func,\partial\Mman} \ \mid \ \func\in\FFF(\Mman,\Pman) \ \}.
\end{align*}

\begin{lemma}
For each $\func\in\FFF(\Mman,\Pman)$ the group $\pi_0\StabilizerIsotId{\func,\partial\Mman}$ belongs to $\classP$.
\end{lemma}
\begin{proof}
By 4) of Theorem~\ref{th:right_action_props} there exist finitely many disjoint subsurfaces $\iSurf_1,\ldots,\iSurf_n \subset \Mman$ each $\iSurf_i$ is diffeomorphic either with $D^2$ or with $S^1\times I$, and such that
\[
\pi_0\StabilizerIsotId{\func,\partial\Mman} \cong
     \mathop{\times}\limits_{i=1}^{n} \pi_0\StabilizerIsotId{\func_{\iSurf_i},\partial\iSurf_i}.
\]
But by Proposition~\ref{pr:pi0StabfB_in_ExtZ} $\pi_0\StabilizerIsotId{\func_{\iSurf_i},\partial\iSurf_i} \in \classP$ for all $i$, whence $\pi_0\StabilizerIsotId{\func,\partial\Mman}\in \classP$ as well.
\end{proof}

For the converse statement we make a remark concerning the structure of groups from $\classP$.
By definition a group $G$ belongs to the class $\classP$ if and only if it can be obtained from the unit group $\{1\}$ by finitely many operations of direct product $\times$ and wreath product $\wrm{m}\ZZZ$ from the top with $\ZZZ$.
We will call such a presentation of $G$ a \emph{$\classP$-presentation}.

A priori a $\classP$-presentation of $G$ is not unique, e.g. $\ZZZ \ \cong \ 1 \wrm{1}\ZZZ \ \cong \ 1  \wrm{3}\ZZZ$.
Given a $\classP$-presentation $\xi_G$ of $G$ denote by $\mu(\xi_G)$ the total number of signs $\times$ and $\wrm{m}\ZZZ$ for some $m\geq1$, used in $\xi_G$.
For example, a group $G=\ZZZ^2 \times (\ZZZ\wrm{4}\ZZZ)$ has a $\classP$-presentation
\[
\xi_G: \ G \ \cong \ (1 \wrm{1}\ZZZ) \times (1 \wrm{1}\ZZZ) \times \bigl( (1 \wrm{1}\ZZZ)  \wrm{4} \ZZZ \bigr).
\]
with $\mu(\xi_G) = 6$.

\begin{lemma}\label{lm:realization_of_pi1}
For each $G\in\classP$ then there exists an $\func\in\Morse(\Mman,\Pman)$ such that
\[ \pi_0\StabilizerIsotId{\func,\partial\Mman} \cong G.\]
\end{lemma}
\begin{proof}
{\bf Case $\Mman=\Disk$ or $S^1\times I$.}
If $G=\{1\}$ is a unit group, we take $\func$ to be a Morse map from 1) or 2a) of Lemma~\ref{lm:pi0StabIsotId_simple_cases} according to $\Mman$.
Then $\pi_0\StabilizerIsotId{\func,\partial\Mman} \cong G  = \{1\}$.

Suppose that we proved our lemma for all groups $A \in\classP$ having a $\classP$-presentation $\xi_A$ with $\mu(\xi_A)<n$ and let $G \in \classP$ be a group having a $\classP$-presentation $\xi_G$ with $\mu(\xi_G)=n$.
It follows from the definition of class $\classP$ that then either
\begin{itemize}
\item[(i)]
there exist $A,B\in \classP$ and $m\geq2$, such that $G \cong A \times (B\wrm{m}\ZZZ)$, where $A$ and $B$ have $\classP$-presentations $\xi_A$ and $\xi_B$ with $\mu(\xi_A), \mu(\xi_B) < \mu(\xi_G)$, or
\item[(ii)]
there exist $A\in \classP$ such that $G \cong A \times \ZZZ$, where $A$ has a $\classP$-presentation $\xi_A$ with $\mu(\xi_A) < \mu(\xi_G)$.
\end{itemize}

First assume that $\Mman=\Disk$.

(i) Suppose $G \cong A \times (B\wrm{m}\ZZZ)$.
Define a Morse function $\varphi:\Mman\to\Pman$, as it is shown in Figure~\ref{fig:pi1_disk_f_g_wr_zn}(a) for $m=3$.
So $\varphi$ has one local minimum $x$ and $m$ local maximums $y_0,\ldots, y_{m-1}$ satisfying $\varphi(y_0) = \cdots= \varphi(y_{m-1})$ and there exists a diffeomorphism $\gdif\in\Stabilizer{\func,\partial M}$ that cyclically interchange these points, i.e.\! $\gdif(y_j) = y_{j+1 \ \mathrm{mod} \ m}$.
Let $\iSurf$ be a $\varphi$-regular disk neighbourhood of $x$, $\nSurf_0$ be a $\varphi$-regular disk neighbourhood of $y_0$, and $\nSurf_j = \gdif^j(y_j)$, $j=1,\ldots,m-1$.
\begin{figure}[ht]
\centerline{\includegraphics[height=4cm]{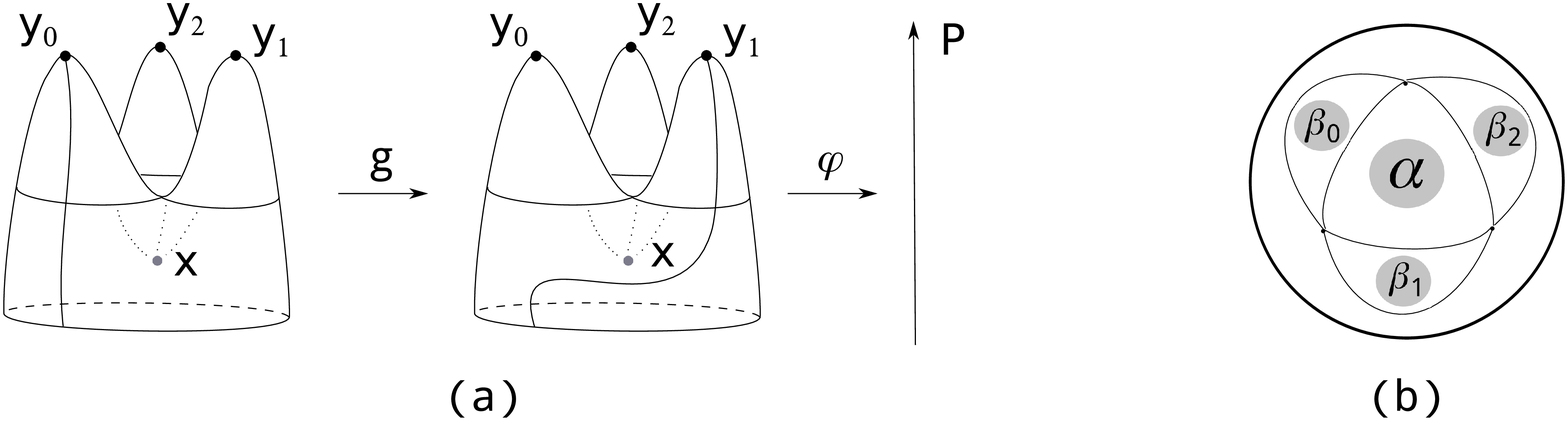}}
\caption{$\Mman=\Disk$. Case (i)}
\protect\label{fig:pi1_disk_f_g_wr_zn}
\end{figure}
As $\mu(\xi_A), \mu(\xi_B) < \mu(\xi_G)$, we have by induction that there exist $\alpha\in\Morse(\iSurf,\Pman)$ and $\beta\in\Morse(\nSurf_0,\Pman)$ such that 
\begin{align*}
A &\cong \pi_0\StabilizerIsotId{\alpha,\partial\iSurf}, &
B &\cong \pi_0\StabilizerIsotId{\beta,\partial\nSurf_0}.
\end{align*}
Not loosing generality, one can assume that $\alpha=\varphi$ in a neighbourhood of $\partial\iSurf$ and $\beta=\varphi$ in a neighbourhood of $\partial\nSurf$.
Replace $\varphi$ with $\alpha$ on $\iSurf$, with $\beta_j = \beta\circ\gdif^{-j}$ on $\nSurf_j$, $j=0,\ldots,m-1$, and denote the obtained new map by $\func$, see Figure~\ref{fig:pi1_disk_f_g_wr_zn}(b).
Then $\func\in\Morse(\Mman,\Pman)$ and it follows from Proposition~\ref{pr:pi0StabfB_in_ExtZ} that 
\[
\pi_0\StabilizerIsotId{\func,\partial\Mman} \cong A \times(B\wrm{m}\ZZZ) \cong G.
\]

(ii) Suppose now $G \cong A \times \ZZZ$.
Define a Morse function $\varphi:\Mman\to\Pman$ having two local maximums $x$ and $y$ such that $\varphi(x) \not= \varphi(y)$, see Figure~\ref{fig:pi1_disk_f_zn}(a).
\begin{figure}[ht]
\centerline{\includegraphics[height=4cm]{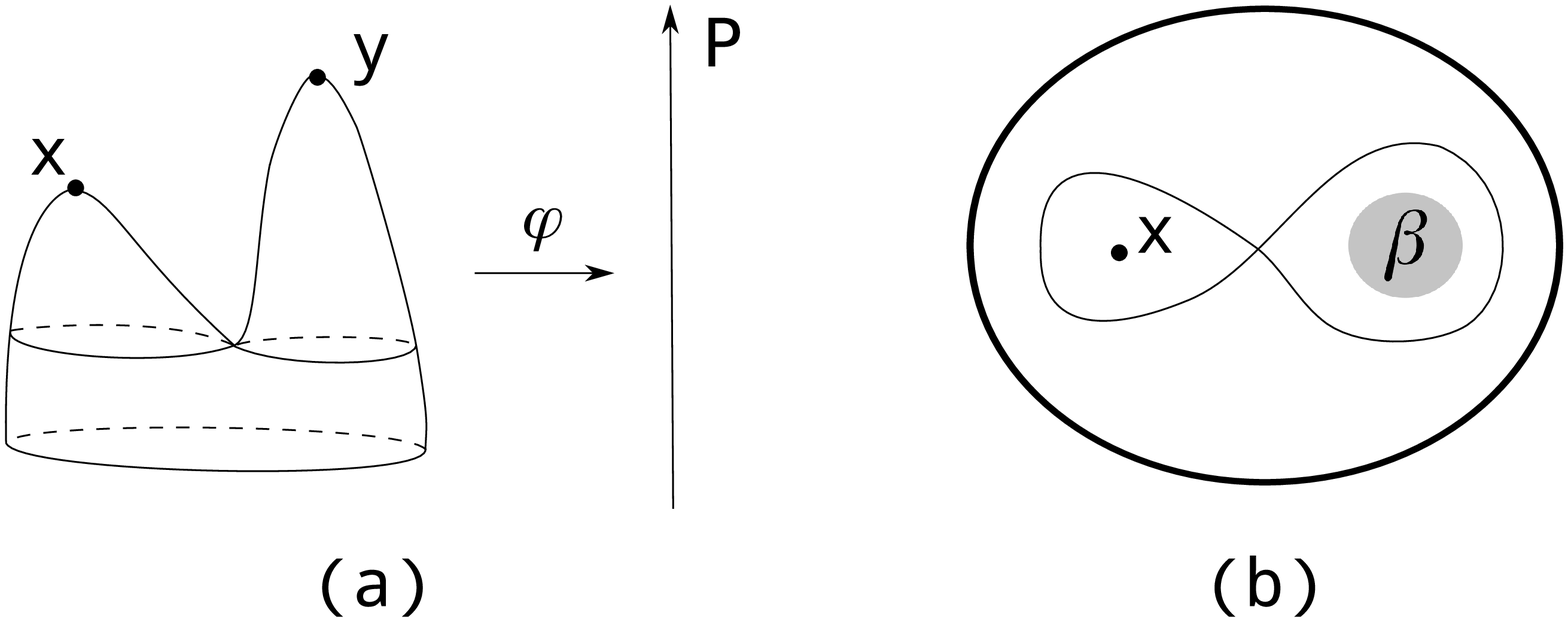}}
\caption{$\Mman=\Disk$. Case (ii)}
\protect\label{fig:pi1_disk_f_zn}
\end{figure}
Let $\nSurf$ be a $\varphi$-regular disk neighbourhood of $y$ such that $\varphi(x) \not\in \varphi(\nSurf)$.
Since $\mu(\xi_A) < \mu(\xi_G)$, it follows by induction that there exist $\beta\in \Morse(\nSurf,\Pman)$ such that $A \cong \pi_0\StabilizerIsotId{\func,\partial\nSurf}$ and $\alpha=\varphi$ near $\partial\nSurf$.
Now replace $\varphi$ with $\beta$ on $\nSurf$ and denote the obtained map by $\func$.
Then $\func\in\Morse(\Mman,\Pman)$ and it follows from Proposition~\ref{pr:pi0StabfB_in_ExtZ} and Lemma~\ref{lm:pi0StabIsotId_simple_cases} 2(a) that
\[
\pi_0\StabilizerIsotId{\func,\partial\Mman} \cong A \times \ZZZ \cong G.
\]

\medskip

For $\Mman=S^1\times I$, the proof of the cases (i) and (ii) is similar to the case of $\Disk$, and is illustrated in Figure~\ref{fig:pi1_cyl}.
We leave the details for the reader.
\begin{figure}[ht]
\begin{center}
\begin{tabular}{cc}
\includegraphics[height=3cm]{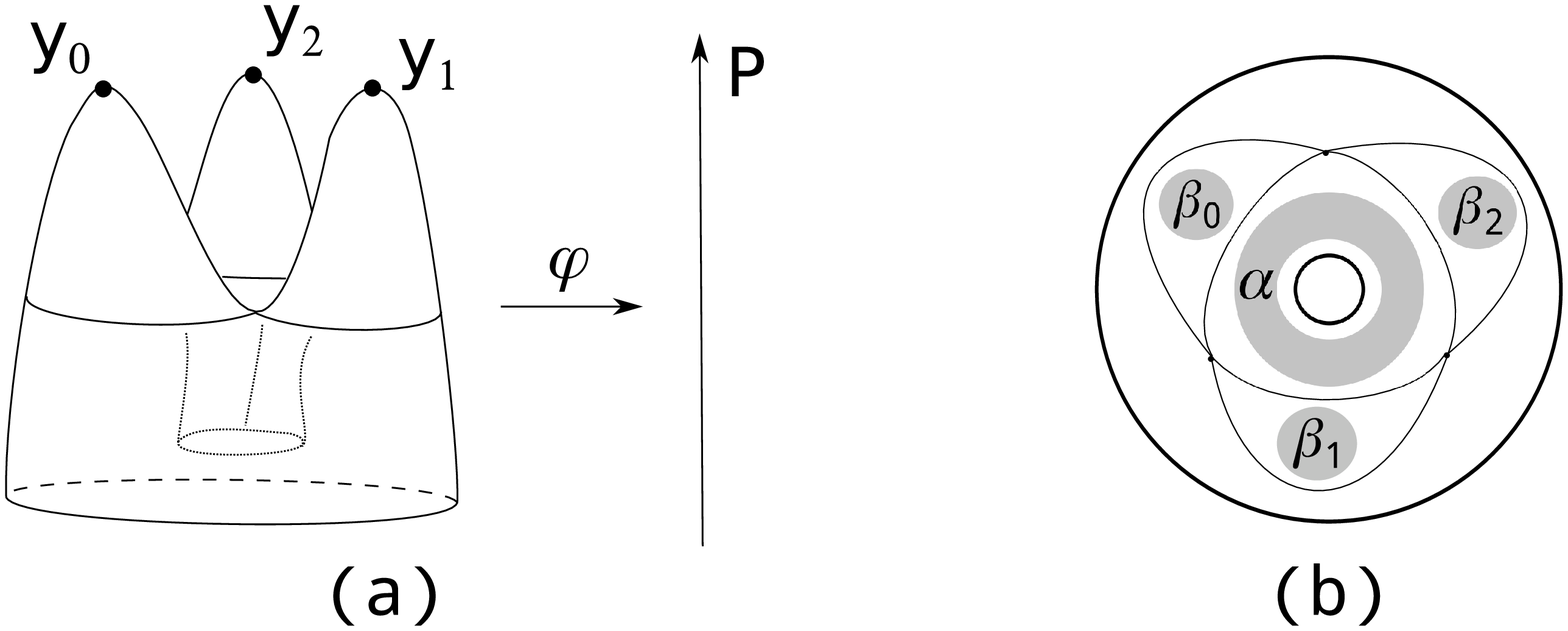} \\
Case (i) \\ \\
\includegraphics[height=3cm]{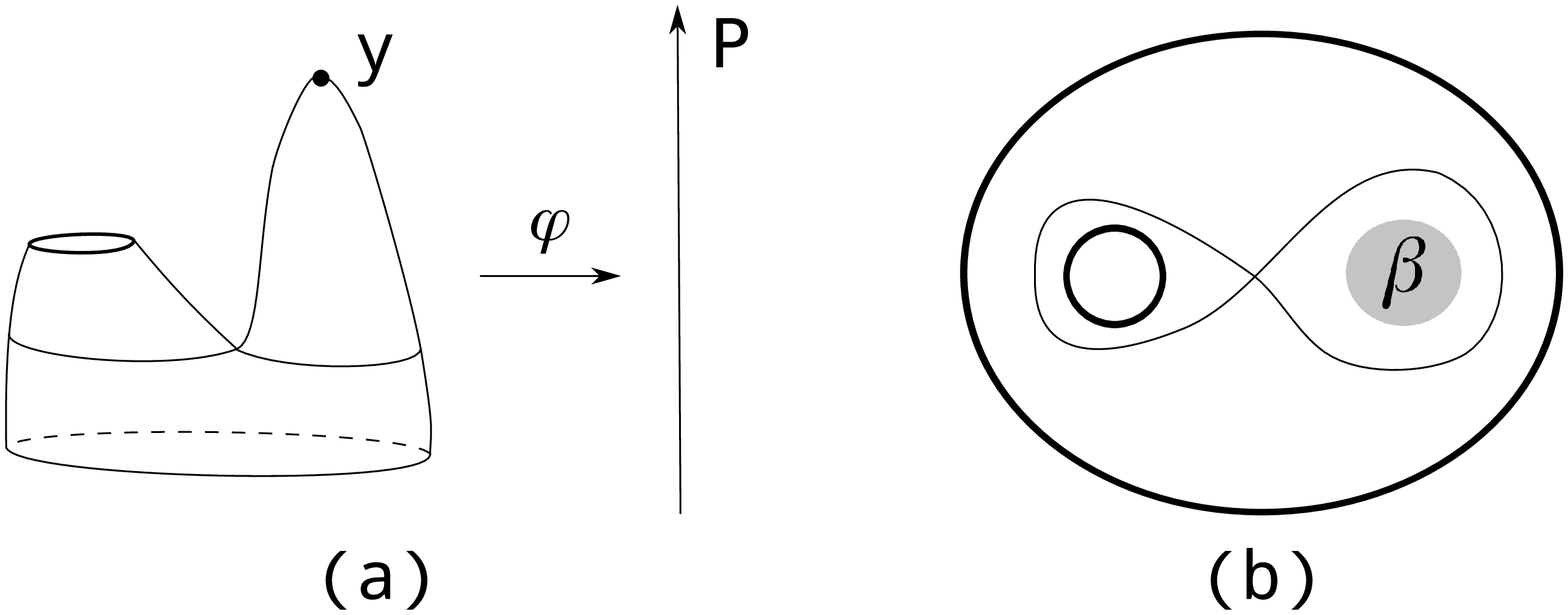} \\
Case (ii)
\end{tabular}
\end{center}
\caption{$\Mman=S^1\times I$}
\protect\label{fig:pi1_cyl}
\end{figure}

Now let $\Mman$ be an arbitrary compact orientable surface distinct from $\Sphere$, $\Torus$, $\Disk$, and $S^1\times I$.
Choose a Morse function $\varphi:\Mman\to\Pman$ such that 
\begin{itemize}
\item
all critical points of $\varphi$ of index $1$ belongs to the same critical level-set of $\varphi$;
\item
the values of $\varphi$ at distinct boundary components and distinct local extremes of $\varphi$ are distinct,
\end{itemize}
see Figure~\ref{fig:pi1_gen_case}.
\begin{figure}[ht]
\centerline{\includegraphics[height=4cm]{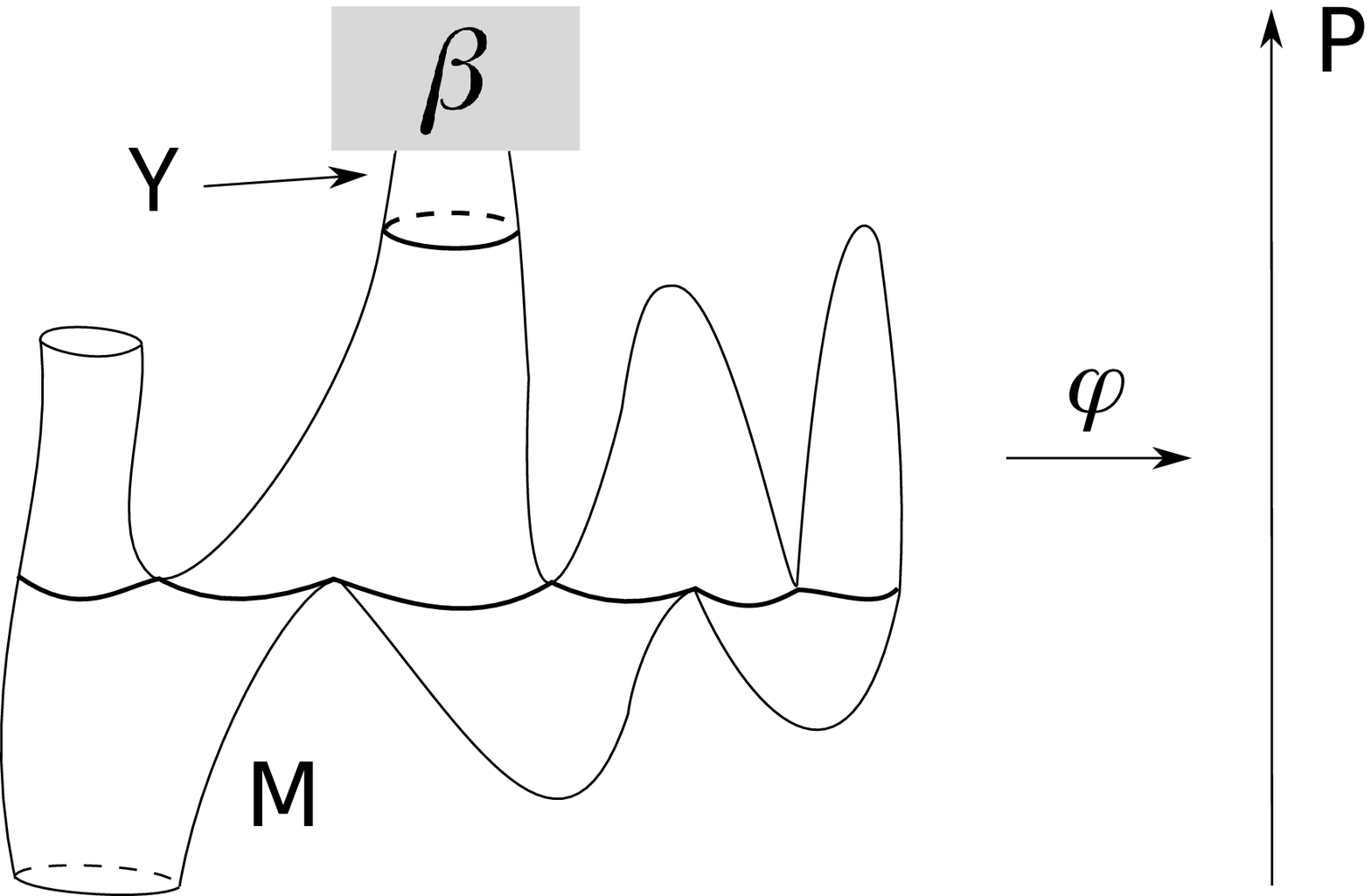}}
\caption{General case}
\protect\label{fig:pi1_gen_case}
\end{figure}
Fix some local extreme $y$ of $\varphi$ and let $\nSurf$ be a $\varphi$-regular disk neighbourhood of $y$.

Let $G\in\classP$.
Since the theorem is already proved for a disk $\Disk \simeq \nSurf$, there exists $\beta\in\FFF(\nSurf,\Pman)$ with $\pi_0\StabilizerIsotId{\beta,\partial\nSurf} \cong G$ and $\beta = \varphi$ in some neighbourhood of $\partial\nSurf$.
Replace $\varphi$ with $\beta$ on $\nSurf$ and denote the obtained map by $\func$.
Then $\func\in\Morse(\Mman,\Pman)$ and it follows from Proposition~\ref{pr:pi0StabfB_in_ExtZ} that $\pi_0\StabilizerIsotId{\varphi,\partial \Mman} \cong \pi_0\StabilizerIsotId{\func,\partial\nSurf}\cong G$.
Lemma~\ref{lm:realization_of_pi1} and Theorem~\ref{th:structure_of_pi1_fOrb} completed.
\end{proof}

\subsection{Acknowledgements}
The author is grateful to B.~Feshchenko of fruitful discussions.


\def\cprime{$'$}
\providecommand{\bysame}{\leavevmode\hbox to3em{\hrulefill}\thinspace}
\providecommand{\MR}{\relax\ifhmode\unskip\space\fi MR }
\providecommand{\MRhref}[2]{%
  \href{http://www.ams.org/mathscinet-getitem?mr=#1}{#2}
}
\providecommand{\href}[2]{#2}

\end{document}